\documentclass[reqno, 11pt]{amsart}
\usepackage[margin=2.5cm]{geometry}
\usepackage{mathrsfs}
\usepackage{amsmath}
\usepackage{amsthm}
\usepackage{amsfonts}
\usepackage{amssymb}
\usepackage{url}
\usepackage{enumerate}
\usepackage[pdftex,bookmarks=true]{hyperref}
  \usepackage[usenames, dvipsnames]{color}
\usepackage{verbatim}
\usepackage{bbm}
\usepackage{pdfsync}   \usepackage{graphicx}
\usepackage{xcolor}
\usepackage{cleveref}

\theoremstyle{plain}
  \newtheorem{theorem}{Theorem}[section]

  \newtheorem{prop}[theorem]{Proposition}
  \newtheorem{proposition}[theorem]{Proposition}
  \newtheorem{fact}[theorem]{Fact}
  \newtheorem{lemma}[theorem]{Lemma}
  
  \newtheorem{cor}[theorem]{Corollary}

  \newtheorem{claim}[theorem]{Claim}
  
\theoremstyle{definition}
  \newtheorem{definition}[theorem]{Definition}
  \newtheorem{remark}[theorem]{Remark}

\numberwithin{equation}{section}
\numberwithin{theorem}{section}

\renewcommand\Re{{\operatorname{Re}}}
\renewcommand\Im{{\operatorname{Im}}}

\newcommand\R{{\mathbb{R}}}
\newcommand\C{{\mathbb{C}}}

\newcommand\Z{{\mathbb{Z}}}

\newcommand\al{\alpha}
\newcommand\la{\lambda}

\newcommand\Be{{\mathbf e}}

\newcommand\BN{{\mathbf N}}

\newcommand\cD{{\mathcal D}}

\newcommand\CG{{\mathcal G}}

\newcommand\CM{{\mathcal M}}

\newcommand\eps{\varepsilon}

\newcommand\rang{\rangle}

\renewcommand\a{\alpha}

\newcommand\bs{\backslash}

\newcommand\im{\operatorname{im}}

\newcommand{\wb}{\overline}

\newcommand\til{\widetilde}

\newcommand\nc\newcommand
\nc\dmo\DeclareMathOperator
\renewcommand\bs\boldsymbol
\definecolor{purple}{rgb}{0.9,0,0.8}
\nc{\nick}[1]{{\color{purple} #1}}
\nc\lf{\lfloor}
\nc{\rf}{\rfloor}
\nc\pr{{\mathbf{P}}}
\nc\lam{{\lambda}}
\nc\tP{{\til P}}
\nc{\xxi}{\xi}
\nc{\ii}{{\sqrt{-1}}}
\dmo{\Leb}{{Leb}}
\nc{\wt}{\widetilde}
\nc{\mLeb}{m_{\Leb}}
\nc{\tran}{{\mathsf{T}}}
\nc{\mom}{{m}}
\nc{\avec}{{\bs a}}
\nc{\bvec}{{\bs b}}
\nc{\uvec}{{\bs u}}
\nc{\vvec}{{\bs v}}
\nc{\wvec}{{\bs w}}
\nc{\ul}{\underline}
\nc{\us}{{\bs{s}}}
\nc{\bbs}{{\bs{s}}}
\nc{\pol}{{P}}
\nc{\pt}{Q}
\nc{\pk}{P}
\nc{\ptt}{{\wt{P}}}
\nc{\pkt}{{\wt{Q}}}
\nc{\tor}{{\R/\Z}}
\nc{\mk}{{m}}
\nc{\mt}{{\wt m}}
\dmo\bad{{bad}}
\nc{\badarcs}{{E_{\bad}}}
\nc{\Cp}{{K_0}}
\nc{\Smooth}{{K}}
\nc{\smooth}{{\kappa}}
\nc{\Spread}{{L}}
\nc{\spread}{{\gamma}}
\nc{\dbd}{{K}}
\nc{\tpol}{{\til\pol}}
\nc{\ali}{{\al_i}}
\nc{\alo}{{\al_1}}
\nc{\alm}{{\al_\mom}}
\nc{\dil}{{L}}
\nc{\mat}{{A}}
\dmo{\row}{{row}}
\nc{\bzeta}{{\bs \zeta}}
\nc{\bxi}{{\bs \xi}}
\nc\height{{\tau}}
\nc\subG{{K_{G}}}
\nc\Points{{\CM}}
\nc\Char{{\Phi}}
\nc\charr{{\phi}}

\numberwithin{equation}{section}

\renewcommand{\phi}{\varphi}
\renewcommand{\emptyset}{\varnothing}

\newcommand{\cE}{\mathcal E}

\newcommand{\cA}{\mathcal A}
\newcommand{\cB}{\mathcal B}

\newcommand{\cG}{\mathcal G}

\newcommand{\bE}{\mathbb E}

\newcommand*{\dprime}{{\prime \prime}}

\newcommand{\dd}{{\rm d}}

\newcommand*{\avg}[1]{\left\langle {#1} \right\rangle}

\newcommand{\norm}[1]{\left\lVert#1\right\rVert}

\def\im#1{\mathrm{Im}\,#1}

\renewcommand{\le}{\leqslant}
\renewcommand{\ge}{\geqslant}

\newcommand{\bR}{\mathbb R}
\newcommand{\bC}{\mathbb C}
\newcommand{\bZ}{\mathbb Z}

\newcommand{\bT}{\mathbb T}
\newcommand{\bN}{\mathbb N}

\newcommand{\bP}{\mathbb P}

\nc\ang{{\theta}}
\nc\dang{{\tau}}
\nc\Walk{{\bs W}}
\nc\walk{{W}}
\nc\pos{{q}}
\nc\PGauss{\bP_{\BN(0,1)}}
\nc\EGauss{\bE_{\BN(0,1)}}

\nc{\NIQ}[1]{{\color{red} \sf $\spadesuit\spadesuit\spadesuit$ Nick: [#1]}}
\nc{\NIC}[1]{{\color{purple} \sf $\spadesuit\spadesuit\spadesuit$ Nick: [#1]}}

\title[Kac polynomials: closest roots to the unit circle]{Universality of Poisson limits for moduli of roots of Kac polynomials}

\author{Nicholas A.\ Cook}
\address{\tiny{Nicholas A.\ Cook, Duke University, Durham, NC 27708, USA}}
\email{nickcook@math.duke.edu}
\author{Hoi H. Nguyen}
\address{\tiny{Hoi H. Nguyen, Department of Mathematics, The Ohio State University, 231 W 18th Ave, Columbus, OH 43210 USA.}}
\email{nguyen.1261@math.osu.edu}
\author{Oren Yakir}
\address{\tiny{Oren Yakir, School of Mathematics, Tel Aviv University, Tel Aviv 6997801, Israel.}}
\email{oren.yakir@gmail.com}
\author{Ofer Zeitouni}
\address{\tiny{Ofer Zeitouni, Department of Mathematics, Weizmann Institute of Science, Rehovot 76100, Israel.}}
\email{ofer.zeitouni@weizmann.ac.il}
\thanks{This project has received funding from the European Research Council (ERC) under the European Union's Horizon 2020 research and innovation programme (grant agreement No. 692452). HN is supported by NSF grant DMS-1752345. OY is supported by ISF Grants 382/15 and 1903/18.}

\begin{document}

	\maketitle
	\begin{abstract}
	We give a new proof of a recent resolution  
	\cite{MS} by Michelen and Sahasrabudhe of 
	a conjecture of Shepp and Vanderbei \cite{shepp_vanderbei}
	that the moduli of
	 roots of Gaussian
	Kac polynomials of degree $n$, centered at $1$ and rescaled by $n^2$,  should form a  Poisson point process. 
	We use this new approach to verify a conjecture from \cite{MS} that the 
	Poisson statistics are in fact universal.
\end{abstract}
\section{Introduction}
Let $\xi_0,\ldots,\xi_n$ be i.i.d. 
random variables and consider the \emph{Kac} polynomial
\begin{equation}
\label{eq:def_of_poly}
f(z) := \sum_{k=0}^{n}\xi_k z^k.
\end{equation}
This paper concerns the typical behavior
of the random zero set $Z(f) := \{z\in \C:f(z)=0\}$.
It is well known that if 
$\bE\log(1+|\xi_0|)<\infty$ then the roots of 
$f$ concentrate uniformly around the unit circle 
$\mathbb{S}^1:=\{|z|=1\}$ as the degree $n$ tends to infinity 
\cite{erdos_turan}, \cite{sparo_sur} 
(see also \cite{hughes_nikeghbali} for a more modern perspective). 
Finer results on this convergence are also known: 
typically, most roots lie inside an annulus of width 
$O(n^{-1})$ around the unit circle (see
\cite{shepp_vanderbei} for
Gaussian coefficients and
\cite{ibragimov_zeitouni} 
for more general choices of coefficients). Microscopic correlation
functions for the Gaussian case 
(also, for multivariate systems of polynomials) 
appear in \cite{BSZ00};
see \cite{DONgV,TV} for a universality result in the case of non-Gaussian
coefficients.

In view of the above, the following question becomes
quite natural: 
\emph{What is the typical distance of the 
	set of roots to the unit circle?} Indeed, this question 
was already suggested by Shepp and Vanderbei 
\cite[Section~6]{shepp_vanderbei}, who conjectured that
for Gaussian coefficients,
the set $\{n^2(1-|z|)\, : \, z\in Z(f)\}$ (identified with its counting  measure) converges 
to a Poisson point process as $n\to \infty$;
their conjecture would imply in particular that
the closest root is typically at distance of order $n^{-2}$ 
from the unit circle. This conjecture was recently confirmed by Michelen and Sahasrabudhe 
\cite{MS}, using a reduction to a point process determined by the
polynomial \textit{on the unit circle} and a
high-order Kac--Rice formula.
\begin{theorem}[{\cite[Theorem~1]{MS}}]
	\label{thm:main_result_gaussians}
	Assume that $f$ is given by \eqref{eq:def_of_poly} 
	with $\xi_i$ i.i.d., 
	real-valued Gaussian of mean zero.
	Then
	\[
	\{n^2(1-|z|)\, : \, z\in Z(f)\}
	\]
	converges in distribution 
	(with respect to the vague topology), as $n\to\infty$,
	to a 
	homogeneous Poisson point process on $\mathbb{R}$
	of intensity $1/12$.
	In particular, for all $x>0$,
	\[
	\lim_{n\to\infty} \bP\left(n^2 \, {\normalfont \text{dist}}\left(\mathbb{S}^1,Z(f)\right)\ge x\right) = e^{-x/6}.
	\]
\end{theorem}
It is natural to guess that the phenomenon described in \Cref{thm:main_result_gaussians} is in fact universal in the sense that the 
theorem
holds true for 
a wide class of coefficients distribution, 
and not just for Gaussians. 
In this regard, it is natural (and also suggested in \cite{MS}) 
to conjecture that
\Cref{thm:main_result_gaussians} holds for
random \emph{Littlewood} polynomials, that is,
when the $\xi_i$'s are i.i.d.
chosen uniformly at random from the set $\{\pm 1\}$. 

The goal of this paper
is multifold. We will give a new proof 
of \Cref{thm:main_result_gaussians},
based on ideas appearing in
our recent works \cite{YZ}, \cite{CNg}. 
In this proof, and in contrast with
\cite{MS},
we do not use the Kac--Rice formula, 
a powerful tool which is hard to apply outside the case of
Gaussian coefficients (or, with considerably more effort,
coefficients whose law possesses a smooth density).
Instead, we work directly
with local linear approximations of  $f$ 
(see Section \ref{sec:proof_method} for more details). Besides 
the intrinsic interest in a different proof, the main
advantage of our approach is that it is more 
suitable 
toward the study of 
universality problem, for which our main result
confirms the prediction. Namely, we show
\begin{theorem}\label{thm:main} \Cref{thm:main_result_gaussians} continues to hold as long as the $\xi_i$ are 
	i.i.d. copies of a non-degenerate, sub-Gaussian, real-valued, zero-mean 
	random variable.
\end{theorem}
In particular, our result 
extends to random Littlewood polynomials. 

\begin{remark} By sub-Gaussian we mean that there exists a constant $c>0$ so that $$\bP(|\xi_0|>t)\leq 2 \exp(-ct^2)$$ for all $t>0$.
	A close investigation of the current method would enable us to allow $|\xi_i|$ to have bounded $k$-th moment for some large constant $k$, however we will not elaborate on this. As it has no bearing on the conclusion, in the rest of the paper we may and will assume that the $\xi_i$'s are normalized to have unit variance.
\end{remark}

\begin{remark} Our methods allow one to consider other related 
  point processes for statistics near the unit circle, see 
  Section \ref{sec-8}.
\end{remark}
\subsection{Heuristics and proof method}
\label{sec:proof_method}

We first explain briefly the Poisson heuristic 
behind the Gaussian case, which is hinted at
in \cite[Section 6]{shepp_vanderbei}. The computations of 
\cite{shepp_vanderbei} and an asymptotic analysis show that
the \textit{expected} number of roots at distance at most
$x/n^2$ from the unit 
circle, denoted $N_x(n)$, is asymptotic to $x/6$.
(The expression in
\cite[Theorem 2]{ibragimov_zeitouni}, based on
the Kac--Rice formula, gives that immediately, and in fact the proof
there can be adapted to cover that asymptotic.)
While computing higher moments
of this quantity via the Kac--Rice formula may be feasible, 
the computation quickly becomes cumbersome, with a major
obstacle being the need to deal with 
short-range correlations and their cancellations. 
Assuming however that such short-range correlations do not affect
the higher order
moments, one notes that for macroscopically separated points
$(z_1,\ldots,z_k)$ on the unit circle, 
the joint density of the random variables $(f(z_i),f'(z_i))_{i=1}^k$ 
nearly factorizes; accepting this factorization, one obtains from
the Kac--Rice formula 
that the expectation of the $k$th moment of $N_x(n)$ converges to the
corresponding one for a Poisson random variable of parameter $x/6$.
From this, the route to a Poisson heuristic is short. In fact, once 
spatial separation between close roots is proved, the Poisson
heuristic is standard.

There are two obstacles for making the heuristic precise, even in the Gaussian case. First, one 
needs to get rid of short range correlations. This is achieved by 
noting, as was done in \cite{KSch,MS}, that in a tubular neighborhood of $z\in \mathbb{S}^1$ 
of diameter $o(1/n)$, the pair $(\Re f(z),\Im f(z))$  
can be approximated by linear curves, and the existence in this 
neighborhood 
of a root of distance $x/n^2$ from the
circle can be decided in terms of the 4-vector 
$W(z):=$ $(\Re f(z), \Im f(z),$ $
\Re f'(z), \Im f'(z))$,  with the result that except for exceptional
values of this vector, only one root
can possibly exist in the neighborhood. 
The second obstacle involves long range correlations: while 
conditioning on a fixed number of values
$W(z_i)_{i\geq 2}$ for macroscopically separated $z_i$ has only a small effect on $W(z_1)$, by analyticity
the values $\{W(z)\}_{z\in \mathbb{S}^1: |z-z_1|> 1/2}$ already   determine
$W(z_1)$ and in particular the existence of a root close to $z_1$.
This is a serious obstacle in applying methods of Poisson convergence
based on second moment methods, e.g. \cite{barbour}.

\subsection{Structure of the paper}
We can now explain our approach to 
\Cref{thm:main} and contrast it with the approach to the Gaussian 
case introduced in \cite{MS}.
There are essentially three major
steps.
\begin{enumerate}
	\item
	As explained above, 
	in the first step we show via \eqref{PPequiv.gauss} and
	\Cref{prop:PPequivalence} that the statistics of 
	$n^2(1-|z|), z\in Z(f)$ near the unit circle can be deduced 
	from the behavior of $(f(z_\alpha),f'(z_\alpha))$ 
	on the unit circle, 
	for $z_\alpha=e^{i\theta_\alpha}$ belonging to a net
	of cardinality $n^{2}/\log^{K_0}n$ (for some large constant $K_0$), and that one needs only consider
	good events, denoted $\mathcal{A}_\alpha$,
	where the linear approximation at $z_\alpha$ 
	is precise enough. Towards the universality result (and in particular to allow discrete distributions for $\xi_0$), for reasons described below 
	we remove from consideration certain ``bad points'' $z_\alpha$ 
	possessing bad arithmetic properties, leaving only
	under consideration \textit{smooth points} $z_\alpha$, see 
	\Cref{def:smooth} and \Cref{prop:exceptional}. (The latter
	removal of bad points is not needed in the Gaussian case.)
	Further, we show (see \Cref{prop:moments:sum} and \Cref{lemma:good_events_are_well_separated}) that one needs only consider $z_\alpha$'s that are separated 
	by distance at least $n^{-1+\eps}$ for some small $\eps>0$.

	\vskip .05in
	\item From Step $1$, for points $z_\alpha$ satisfying the good event,
	the location of the root close to $z_\alpha$ can be written as 
	$Z_\alpha=n^2\rho_\alpha=n^2 \rho_\alpha(f(z_\alpha),f'(z_\alpha))$, see
	\eqref{eq:solution_to_linear_system} and \eqref{eq:def_mu_f}.
	Then, the $k$th moment of $N_x(n)$ can be determined in terms of the
	joint distribution of $(Z_{\alpha_j})_{j=1}^k$ for well-separated, smooth
	$z_{\alpha_i}$. In the Gaussian case, these are almost independent (see
	the computation in Section   \ref{subsec-control_covariance} for a quantitative estimate), and the moments
	factor, see \Cref{lem:decorr}, yielding \Cref{thm:main} in the Gaussian case. No notion of smoothness of the $z_\alpha$'s is needed for that computation.
		\vskip .05in
	\item The key tool for obtaining universality is a result imported from \cite{CNg} giving a quantitative local CLT for the joint distribution of $W(z_1),\dots, W(z_k)$ for any fixed number of points $z_1,\dots, z_k\in \mathbb{S}^1$  that are both smooth and spread
	(see \Cref{thm:box}).
	Combined with small ball 
	estimates, also borrowed from \cite{CNg}
	(see \Cref{thm:smallball}),
	this leads to a comparison of probabilities between the Gaussian and general
	cases, culminating with \Cref{prop:compare}.
		\vskip .05in
	We note in passing that the local CLT borrowed
	from \cite{CNg}, arguably the technically 
	most challenging component used in our 
	proof, is in turn a significant generalization
	of a result from \cite{DNN} to the multivariate case.
\end{enumerate}

\subsection{Background}
We compare our result with existing 
literature.  Konyagin and Schlag \cite{konyagin,KSch}
showed that, for random Kac polynomials as 
general as in Theorem \ref{thm:main}, 
with high probability the complex roots of $f(z)$ 
cannot be within distance $o(1/n^2)$ of the unit circle. 
This is consistent with our theorem, 
but the method of \cite{KSch}, 
on the other hand, cannot be used to 
study asymptotic statistics. More relevant to our universal statistics,  \cite[Theorem 5.7]{TV} of Tao and Vu, or \cite[Theorem 2.3]{DONgV} of Do et 
al. established universality for the correlation functions
of complex roots of $f(z)$ within balls of radius $\Theta(1/n)$ near the unit circle; it seems that
the techniques and results there cannot be applied to our problem because the region under current consideration is at a finer scale and too close to the circle.

We already mentioned that  \Cref{thm:main_result_gaussians}, the
Gaussian case of \Cref{thm:main}, was already proved in \cite{MS},
using a method based on the higher
Kac--Rice formula, which seems ill-suited 
to computations in the non-Gaussian case. It does not seem
straightforward to apply the local CLT comparison
to their computation and obtain directly \Cref{thm:main} from their
\Cref{thm:main_result_gaussians}.

As discussed above, the method of proof of \Cref{thm:main} borrows from the
recent \cite{CNg,YZ}, that dealt with the variable $m_n^*=\min_{z\in \mathbb{S}^1} |f(z)|$; we refer to these articles
for historical background. (That the problems are related appears already in the work of Konyagin and Schlag \cite{KSch}.)
Step 1 in the sketch above is similar to the analogous step in
\cite{YZ}, with the variable $Z_\alpha$ here differing from the variable
$Z_\alpha$ in \cite{YZ,CNg} by a factor of $1/|f'(z_\alpha)|$. 
In the case
of $m_n^*$, with $\xi_i$ Gaussian,
the collection $(Z_{\alpha_j})_{j=1}^k$ 
for separated $z_{\alpha_j}$ has the property that writing
$\xi_i=\sqrt{1-\epsilon^2} \xi_i'+\epsilon \xi_i''$ with $\xi_i',\xi_i''$
independent and $\epsilon$ small ($n$ dependent), one has that $Z_{\alpha_j}=Z'_{\alpha_j}+n_j$
where $Z_{\alpha_j}'$ are the same as $Z_{\alpha_j}$ 
except that $\xi_i'$ are used instead of $\xi_i$, and 
the $n_j$ are (asymptotically) independent of each other and of the $(Z_{\alpha_j}')$. 
From this, an application of an 
invariance result due to Liggett \cite{liggett} (see also \cite{CGS}) gave the Poisson limit in
the Gaussian case. In the case considered in this paper, 
with the definition of $Z_\alpha$ as in 
\eqref{eq:solution_to_linear_system} and \eqref{eq:def_mu_f},
one can repeat the computation, 
but we cannot show a-priori that
the $n_j$'s are asymptotically independent
of the $Z_{\alpha_j}'$. Thus, Step 2 here differs significantly from
the proof in \cite{YZ} (which is used as basis for comparison in \cite{CNg}).

	\subsection*{Notation}
	We will assume $n\to\infty$ throughout the paper. We write $X=O(Y)$, $X\ll Y$ or $Y\gg X$ if $|X|\leq C|Y|$ for some absolute constant $C>0$ which does not depend on $n$. We write $X\asymp Y$ or $X=\Theta(Y)$ if $X\ll Y$ and $Y\ll X$. We write $X=o(Y)$ or $Y=\omega(X)$ if $X/Y\to 0$ as $n\to\infty$. In what follows, $\norm{\cdot}_{\bR/\bZ}$ is the distance to the nearest integer and $\dd m = \dd \mLeb$ is the Lebesgue measure. The volume of a direct product of intervals $Q=I_1\times\cdots\times I_d\subset \R^d$ is denoted $|Q|=|I_1|\times\cdots\times |I_d|$.
	We write
	$\BN(a,b)$ for the real Gaussian law with mean $a$ and variance $b$.
	
	For random variables $X$ and $Y$, we
	write $X\stackrel{\mbox{\scriptsize{\rm law}}}{=} Y$ if they are identically distributed. For a sequence of random variables $X_n$, we write $X_n\xrightarrow{\ d \ } X$ if
	$X_n$ converges in distribution to $X$ as $n\to\infty$. Finally, for 
	$N\in \bN$ we write $[N] := \left\{0,1,\ldots,N\right\}$.
	
	Sometimes we write $\bP_{\BN(0,1)}(\cdot), \bE_{\BN(0,1)}(\cdot)$ to emphasize that the model under consideration is (standard) gaussian; in the general case we will drop these subscripts.
	
	\section{Reduction to the unit circle}
	\label{sec:reduction_to_unit_circle}
	For $t\in [0,\pi]$\footnote{Because $f(\wb{z}) = \wb{f(z)}$, it suffices to work with the upper half plane.} we write
	\[
	\frac{1}{\sqrt{n}}f\left((1+\rho)e^{it}\right) = X(\rho,t) + i Y(\rho,t)
	\]
	where
	\begin{align*}
	&X(\rho,t) = \frac{1}{\sqrt{n}}\sum_{k=0}^{n} \xi_k (1+\rho)^k \cos(kt), \qquad Y(\rho,t) = \frac{1}{\sqrt{n}}\sum_{k=0}^{n} \xi_k (1+\rho)^k \sin(kt).
	\end{align*}
	On the unit circle we denote
	\begin{equation}
X(t) := X(0,t) \mbox{ and } Y(t):= Y(0,t).
\end{equation}

	For $S\subset \bR$, we define the point process (random measure) $\nu_f$ of ``close-roots" (counting multiplicity) by the formula 
	\begin{equation*}
	\nu_f(S):= \left|\left\{z\in Z(f) \mid n^2(|z|-1)\in S,\ \im(z)\ge 0 \right\}\right|.
	\end{equation*}
	In fact, Theorem \ref{thm:main_result_gaussians} exactly states that $\nu_f$ converges in distribution to a Poisson point process on $\bR$. As mentioned in the introduction, the goal of this section is to introduce a new point process that is equivalent to $\nu_f$ but is easier to work with -- this we do in what follows.
	
	We let
\begin{equation*}	\label{def:N}
N := \bigg\lfloor \frac{n^{2}}{\log^{K_0} n} \bigg\rfloor
\end{equation*}
for an absolute constant $K_0>8$ that will be taken sufficiently large over the course of the proof,
	and let
	\begin{equation}
	\label{eq:defintion_of_sample_pts}
	\ang_\al:= \frac{\pi \alpha}{N},\quad \alpha = 0,\ldots,N.
	\end{equation}
	Also set $I_\alpha := [\ang_\al-\pi/2N,\ang_\al+\pi/2N]$ to be a covering of $[0,\pi]$ by intervals of equal length. We consider annular domains 
	\[
	C_\al := \left\{z\in \bC  : (1-|z|)\in \left[-\frac{\log n}{n^2},\frac{\log n}{n^2}\right] , \ \text{arg}(z)\in I_\al \right\}.
	\]
	Since $\mLeb(C_\al)$ is small, we expect that $f(z) \approx f(e^{i\ang_{\al}}) + (z-e^{i\ang_{\al}}) f^{\prime}(e^{i\ang_{\al}})$ for all $z\in C_\al$. 
	(We will actually perform the linearization in the $(\rho,t)$-plane.)  In particular, this linear approximation will determine, with high probability, whether or not $f$ has a root inside the set $C_\al$, and will also give an approximation to its location if it exists.
	
	\subsection{Linear approximation}
	
	For every $\al\in [N]$ we sample the real and imaginary parts of $\frac1{\sqrt{n}}f$ and their radial derivatives at angle $\ang_\al$, i.e.
	\begin{equation}
\label{eq-XY}
\left(X(\ang_\al), Y(\ang_\al), X^\prime(\ang_\al), Y^\prime(\ang_\al) \right).
	\end{equation}
	By applying a $2$-dimensional Taylor approximation in the variable $z=(1+\rho)e^{it}$, we arrive at the following linear system:
	\begin{equation}
	\label{eq:system_of_two_equation_for_root}
		\begin{pmatrix}
		X(\rho,t) \\ Y(\rho,t)
		\end{pmatrix} = \begin{pmatrix}
		X(\ang_\al) \\ Y(\ang_\al)
		\end{pmatrix} + \begin{pmatrix}
		X^\prime(\ang_\al) & Y^\prime(\ang_\al) \\ Y^\prime(\ang_\al) & -X^\prime(\ang_\al)
		\end{pmatrix} \cdot\begin{pmatrix}
		t-\ang_\al \\ \rho
		\end{pmatrix} + \text{error},
	\end{equation}
	where we expect the error term to be small inside $C_\al$. 
	
	We will denote by 
	\begin{equation}
		\label{eq:def_of_linear_approximation}
		F_\al(t,\rho):= \begin{pmatrix}
		X(\ang_\al) \\ Y(\ang_\al)
		\end{pmatrix} + \begin{pmatrix}
		X^\prime(\ang_\al) & Y^\prime(\ang_\al) \\ Y^\prime(\ang_\al) & -X^\prime(\ang_\al)
		\end{pmatrix} \cdot \begin{pmatrix}
		t - \ang_\al \\ \rho
		\end{pmatrix}
	\end{equation}
	as the affine map that approximates $\frac{1}{\sqrt{n}}f$ 
	inside $C_\al$. Sometimes for convenience we also write $F_\al(z)$ for $F_\al(\ang,\rho)$.
	
	By setting $f(z) = 0$ and neglecting the error term, the system \eqref{eq:system_of_two_equation_for_root} gives a local candidate for the closest root by
	\begin{align}
		\label{eq:solution_to_linear_system}
		&\rho_\al := \frac{X^\prime(\ang_\al)Y(\ang_\al)-X(\ang_\al)Y^\prime(\ang_\al)}{X^\prime(\ang_\al)^2 + Y^\prime(\ang_\al)^2} 
		 \\\nonumber 
		&\dang_\al :=\ang_\al - \frac{X(\ang_\al)X^\prime(\ang_\al)+Y(\ang_\al)Y^\prime(\ang_\al)}{X^\prime(\ang_\al)^2 + Y^\prime(\ang_\al)^2} .
	\end{align}
	Indeed, it is straightforward to check that $F_\al\left(\dang_\al - \ang_\al , \rho_\al\right) = 0$, and so $\rho_\al$ given in \eqref{eq:solution_to_linear_system} is a linear prediction for the radial position of the closest root to the point $e^{i\ang_\al}$ based on the values of $X(\theta_\al),Y(\theta_\al),X'(\theta_\al),Y'(\theta_\al)$.
	
	Now, we define a new, thinned point process $\mu_f$ which retains points $\rho_\al$ provided that some good event holds. Set
	\begin{align}
	\label{eq:defA}
	\cA_\al := &\cA_\al^\prime \cap \cA_\al^\dprime,\qquad
	\mbox{\rm where} \nonumber \\
	&\cA_\al^\prime := \left\{\dang_\al \in I_\al, \ n^2|\rho_\al| \leq \log n \right\} = \left\{(1+\rho_\al)e^{i\dang_\al} \in C_\al \right\} 
	\\ & \cA_\al^\dprime: =\left\{ |X(\ang_\al)|,|Y(\ang_\al)| \leq n^{-2/3}, \ |X^\prime(\ang_\al)|,|Y^\prime(\ang_\al)| \in[ n\log^{-2K_0}n,n\log^2n] \right\}.\nonumber
	\end{align}	
	The event $\cA_\al^\prime$ implies that the linear approximation predicts a root inside $C_\al$, while the event $\cA_\al^\dprime$ is just typical for such a prediction and tells us that the linear system \eqref{eq:solution_to_linear_system} is non-degenerate. 
	
	 With the above notations, we now define our approximating point process as
	\begin{equation}
	\label{eq:def_mu_f}
	\mu_f:= \sum_{\al =0}^{N} \delta_{Z_\al} \qquad \text{where}\quad Z_\al:= n^2\rho_\al \cdot \mathbf{1}_{\cA_\al} + \infty \cdot \mathbf{1}_{\cA_\al^c}.
	\end{equation}
	And so, $\mu_f$ is a point process in $\bR$ whose values are determined only by the samples of $X,Y,X',Y'$ on the points $\{\ang_\al\}_{\al= 0}^{N}$. 
	We will show that $\mu_f$ serves as a good approximation to $\nu_f$. In particular, our argument shows that for any compact interval $U\subset \bR$,
\begin{equation}	\label{PPequiv.gauss}
		\lim_{n\to\infty} \bP_{\BN(0,1)}\left(\mu_f(U) = \nu_f(U) \right) = 1.
\end{equation}
\subsection{Smooth points} To extend this comparison to general coefficients we will need to remove points $Z_\al$ coming from certain ``bad arcs''. 

\begin{definition}[Smooth points]
\label{def:smooth}
For $\Smooth>0$, we say a point $t\in [0,\pi]$ is \emph{$\Smooth$-smooth} if 
\[
\Big\| \frac{p_0 t}{\pi }\Big\|_{\R/\Z} > \frac{\Smooth}{n}\qquad \forall\; p_0\in [1,K+1]\cap \Z .
\]
We say a tuple $(t_1,\dots, t_\mom)$ is $\Smooth$-smooth if $t_r$ is $\Smooth$-smooth for each $1\le r\le \mom$.
\end{definition}

Letting
\begin{equation}	\label{def:badarcs}
\badarcs = \Big\{ t\in [0,\pi]: \text{ $t$ is not $n^{\kappa}$-smooth} \Big\}
\end{equation}
for some sufficiently small $\kappa>0$ that we choose later, we note that $\mLeb(\badarcs) = O(n^{2\kappa-1})$. We now define modified point processes	
\begin{equation}
	\label{eq:def_mu_f_sharp}
	\mu_f^\sharp:= \sum_{\al: \ang_\al \notin  \badarcs} \delta_{Z_\al}\,, \qquad\qquad \mu_f^\flat:= \mu_f -\mu_f^\sharp
	\end{equation}
and 
	\begin{equation}
	\label{eq:def_nu_f}
	\nu_f^\sharp:= \sum_{\al:\theta_\al\notin\badarcs} \sum_{z\in Z(f)\cap C_\al} \delta_{n^2(|z|-1)}\,,  
	\qquad\qquad \nu_f^\flat := \nu_f - \nu_f^\sharp.  	
	\end{equation}
The following lets us focus on potential angles that are smooth. 

\begin{prop}[Ruling out bad arcs]\label{prop:exceptional}
With probability $1-O(n^{-\kappa/4})$ we have 
\begin{equation}	\label{badarcs.i}
\mu_f^\flat(\R) =0.
\end{equation}
Furthermore,
\begin{equation}	\label{badarcs.ii}
\bP\Big(\, \exists \ang \in \badarcs, z\in Z(f) : |z - e^{i \ang}| \le n^{-3/2}\Big)=O(n^{-\kappa/4}).
\end{equation}
In particular, with probability $1-O(n^{-\kappa/4})$, 
$$\nu_f^\flat(\R)=0.$$	
\end{prop}

We postpone the proof of \Cref{prop:exceptional} to \Cref{sec:exceptional}. 
We can now state the modified version of the comparison \eqref{PPequiv.gauss} allowing us to pass to the point process $\mu_f^\sharp$ in the general case.

\begin{proposition}[Passage to the unit circle]
		\label{prop:PPequivalence}
		For any interval $U\subset \bR$ we have
		\[
		\lim_{n\to\infty} \bP\left(\mu_f^\sharp(U) = \nu_f^\sharp(U) \right) = 1.
		\]
\end{proposition}
	For $x\in \bR$ and $k\ge1$ denote
	\[
	(x)_k :=\max\{0, x(x-1)\cdot \ldots \cdot (x-k+1) \}.
	\] 
	Recall that for a random variable $Z$ which has $\text{Poisson}(\lambda)$ distribution we have $\bE[(Z)_k] = \lambda^k$ for all $k\ge1$. The following states that the factorial moments for $\mu_f^\sharp$ match those of a Poisson process in the limit.
	\begin{prop}[Moments]
		\label{prop:moments:sum}
		Let $U\subset \bR$ be a finite union of compact intervals. Then
		\begin{equation}\label{eq:factorial_moment_of_mu_f_k}
		\lim_{n\to\infty} \bE\left[\left(\mu_f^\sharp(U)\right)_{k}\right] = \left(\frac{|U|}{12}\right)^{k} \qquad \text{for all}  \ k\ge 1.		\end{equation}	
	\end{prop}	
	Assuming that Propositions \ref{prop:exceptional}, \ref{prop:PPequivalence} and \ref{prop:moments:sum} holds we can prove our main result.
	\begin{proof}[Proof of \Cref{thm:main}]
		By Propositions \ref{prop:exceptional} and \ref{prop:PPequivalence} it suffices to show that $\mu_f^\sharp$ has a limiting Poisson distribution as $n\to\infty$. Indeed, Proposition \ref{prop:moments:sum} together with the classical moment theorem \cite[Theorem~3.3.26]{durrett} implies that for any finite union of compact intervals $U\subset \bR$
		\[
		\mu_f^\sharp\left(U\right) \xrightarrow{\ d \ } \text{Poisson}\left(\frac{|U|}{12}\right).
		\]
		Together with a theorem of Kallenberg \cite[Theorem~4.7]{kallenberg}, this implies that the sequence of point processes $\mu_f^\sharp$ converge in the vague topology to a Poisson point process of intensity $1/12$, as desired.
	\end{proof}
	\subsection{Joint distribution over spread points}
	Expanding the factorial moments \eqref{eq:factorial_moment_of_mu_f_k} leads us to consider the joint events that $k$ different samples of our polynomial $f$ on the unit circle contributed a point to $\mu_f^\sharp$ which lie in $U$. Since we already imposed a smoothness assumption on the sample points, to compute the desired probabilities we will require all of the points to be separated from one another, in the following sense:
	\begin{definition}[Spread tuples]
		\label{def:spread}
		We say that $\bs t = (t_1,\dots,t_k)\in [0,\pi]^k$ is \emph{$\gamma$-spread} if
		\begin{equation}
		|t_i-t_j|\ge \frac{\spread}{n}
		\qquad \forall \, -1\le i<j \le k
		\end{equation}
		where we put $t_{0}:= 0$, $t_{-1}:= \pi$.
		(We note that this is different from the definition in \cite{CNg} for general $(t_1,\dots, t_k)\in \R^k$, but when specialized to $[0,\pi]^k$ they are equivalent up to modification of $\gamma$ by a constant factor.)
		Note this definition includes $k=1$, with a single point $t\in[0,\pi]$ being $\gamma$-spread if it is distance at least $\gamma/n$ from $0$ and $\pi$. In particular, if $t$ is $1$-smooth then it is 1-spread. 
	\end{definition}
	The main step towards the proof of \Cref{prop:moments:sum} is the following:
	\begin{prop}[Main term]
		\label{prop:moments}
		Let $U\subset \bR$ be a finite union of compact intervals and fix a $k$-tuple of indices $(\al_1,\ldots\al_k) \in [N]^k$. Assume that for some $\kappa,\eps>0$ the tuple $(\ang_{\al_1},\ldots,\ang_{\al_k})$ is $n^\kappa$-smooth and $n^\eps$-spread, where $\ang_{\al}$ is given by \eqref{eq:defintion_of_sample_pts}. Then,
		\begin{equation*}
			\bP\left(Z_{\al_1}\in U, \ldots,\ Z_{\al_k}\in U \right) = \left(\frac{|U|}{12 N}\right)^{k} + o\left(N^{-k}\right)
		\end{equation*}
		where the rate of convergence depends on $k,\kappa,\eps$ and $K_0$.
		\end{prop}
	We will prove Proposition \ref{prop:moments} in Section \ref{sect:moments}, and with it conclude Proposition \ref{prop:moments:sum} in Section \ref{sect:moments:sum}. Furthermore, some of the tools which we develop in Section \ref{sect:moments} will be helpful for us to prove Proposition \ref{prop:PPequivalence}, which is given in Section \ref{sect:PPequivalence}.

\begin{remark}[Gaussian case]
\label{rmk:gaussian.case}
For the Gaussian case we can skip some steps above and in the proofs of the assumed Propositions \ref{prop:exceptional}, \ref{prop:PPequivalence}, \ref{prop:moments:sum} and \ref{prop:moments}, and do not need all of the tools gathered in \Cref{sec:prelim}. 
In particular:
\begin{itemize}
\item One does not need to modify the processes in \eqref{eq:def_mu_f_sharp}, \eqref{eq:def_nu_f} -- one only needs to remove points coming from a small neighborhood of the real line. This means we only need an easier version of \Cref{prop:exceptional} that only rules out any contribution from almost-real zeros (which are covered by $\badarcs$). See for instance \cite[Lemma~4.3]{YZ} or \cite[Lemma~13]{MS}.
	\vskip .05in
\item In the proof of \Cref{prop:moments} we can skip the application of \Cref{prop:compare} giving quantitative comparison between the Gaussian case and the general case (this is the only place it is applied). 
	\vskip .05in
\item The proofs of Propositions \ref{prop:PPequivalence} and \ref{prop:moments:sum} remain essentially unchanged -- we just need small ball estimates, which in the Gaussian case are immediate from the boundedness of the Gaussian density, whereas in the general case we apply \Cref{prop:domain.UB}. 
\end{itemize}
We note that the only places in the paper where we deal specifically with the Gaussian distribution are in the proof of \Cref{prop:moments}  (specifically, \Cref{lem:gaussian.computation}) and in \Cref{lem:decorr} on decorrelation of the field at large separations. 

\end{remark}

\section{Some supporting lemmas}
\label{sec:prelim}

\subsection{Control on derivatives}

We first start with a standard result. 

\begin{lemma}\label{lemma:bound:derivative} Let $A,A'$ be positive constants where $A$ is sufficiently large. For any $ \la \ge A \sqrt{\log n}$, with probability at least $1- \exp(-\Theta(\la ^2))$ we have that 
$$\max_{||z|- 1| \le A'/n} |f'(z)|, |\partial f /\partial \theta |, |\partial f /\partial \rho | \le  \la n^{3/2}.$$
We also have similar bounds for second order derivatives that
$$\max_{||z|- 1| \le A'/n} |f''(z)|,| \frac{\partial^2 f}{ \partial^l \theta \partial^{2-l} \rho} |\ll \la n^{5/2}, l=0,1,2.$$
\end{lemma}
In fact the above holds over $|z| \le 1+A'/n$ but we don't need this fact here.

Before proving the result, recall that  $f(z)= \sum_{k=0}^{n}\xi_k z^k = f\left((1+\rho)e^{i\ang}\right) =\sum_{k=0}^n \xi_k (1+\rho)^{k} \cos(k\theta) + i \sum_{k=0}^n \xi_k (1+\rho)^{k} \sin(k\theta) = \sqrt{n}(X(\rho,\ang) + i Y(\rho,\ang))$. Hence
\begin{align*}
|\frac{\partial  \sqrt{n}X(\rho,\ang)}{\partial \theta}| &= |\sum_{k=0}^n k\xi_k (1+\rho)^{k} \sin(k\theta)| =  |\sum_{k=0}^n (k+1) (k\xi_k/(k+1)) (1+\rho)^{k} \sin(k\theta)|\\
&=  |\sum_{k=0}^n (k+1) \xi_{k+1}' (1+\rho)^{k} \sin(k\theta)| =  |\Re(g'(z))|, 
\end{align*}
where $\xi_{k+1}' = k\xi_k/(k+1)$ and $g(z) = \sum_{k=0}^n \xi_{k+1}' z^{k+1}$, and similarly
\begin{align*}
|\frac{\partial  \sqrt{n}X(\rho,\ang)}{\partial \rho}| &= |\sum_{k=0}^n k\xi_k (1+\rho)^{k-1} \cos(k\theta)| =  |\frac{1}{1+\rho}\sum_{k=0}^n (k+1) (k\xi_k/(k+1)) (1+\rho)^{k} \cos(k\theta)|\\
&=  |\frac{1}{1+\rho}\sum_{k=0}^n (k+1) \xi_{k+1}' (1+\rho)^{k} \sin(k\theta)|=  |\frac{1}{1+\rho} \Im (g'(z))|.
\end{align*}
One can also have similar expression for $\frac{\partial  \sqrt{n}Y(\rho,\ang)}{\partial \theta}, \frac{\partial  \sqrt{n}Y(\rho,\ang)}{\partial \rho}$. Furthermore, by the same way, the second derivatives $\frac{\partial^2 X(\rho,\theta)}{ \partial^l \theta \partial^{2-l} \rho},\frac{\partial^2 Y(\rho,\theta)}{ \partial^l \theta \partial^{2-l} \rho} $ can be related to $h''(z)$, where $h$ is a polynomial of degree $n+2$ and the coefficients are independent of mean zero, variance almost one, and subgaussian. Hence it suffices to establish Lemma \ref{lemma:bound:derivative} only for $\max_{||z|- 1| \le A'/n|} |f'(z)|$ and $\max_{||z|- 1| \le A'/n} |f''(z)|$.

\begin{remark}\label{rmk:CR}
One notes from the Cauchy--Riemann equations in polar form (or Taylor expansion) that for $f$ analytic in a neighborhood of $z_0=r_0e^{i\theta_0}$, 
\[
|\frac\partial{\partial r}f(re^{i\theta_0})|_{r=r_0} \asymp  |f'(r_0e^{i\theta_0})|\,,
\qquad
|\frac\partial{\partial\theta}f(r_0e^{i\theta})|_{\theta=\theta_0} \asymp r_0 |f'(r_0e^{i\theta_0})|\,.
\]
In particular we have $\frac1{\sqrt{n}}|f'((1+\rho)e^{i\theta}|\asymp (\frac{\partial X}{\partial\theta}(\rho,\theta)^2 + \frac{\partial Y}{\partial\theta}(\rho,\theta))^{1/2}$ uniformly over $(\rho,\theta)\in [-1/2,1/2]\times[0,2\pi]$, say, and we have a similar equivalence for second derivatives.
\end{remark}

\begin{proof}[Proof of Lemma \ref{lemma:bound:derivative}] For the first derivatives, it suffices to show that with probability at least $1- \exp(-\Theta(\la ^2))$ we have $\max_{||z|-1| \le A'/n}|f'(z)| \le \la n^{3/2}$. By the maximum principle, it thus suffices to work with $\max |f'(z)|$ on either $\{(1+A'/n)e^{i \theta}, \theta \in \bT\}$ or $\{(1-A'/n)e^{i \theta}, \theta \in \bT\}$. Without loss of generality, let us focus only on the real part of $f'(z)$ on the larger circle. In other words we will show that  with probability at least $1- \exp(-\Theta(\la ^2))$ we have  $\max_{\theta \in \bT} |\sum_{k=0}^n \xi_kk (1+A'/n)^{k-1} \cos((k-1)\theta)| \le \la n^{3/2}$. For short denote just for the proof by $h(\ang):= n^{-3/2} \sum_{k=0}^n\xi_k k (1+A'/n)^{k-1} \cos((k-1)\theta)$. Denote $\lVert h \rVert_\infty := \sup_{\ang\in \bT} |h(\ang)|$ and let $\tilde{\ang}\in \bT$ be such that 
		$$\lVert h \rVert_\infty = |h(\tilde{\ang})|.$$ 
		Since $h$ is a trigonometric polynomial of degree $n$, we can apply the Bernstein inequality and see that for all $|\ang-\tilde{\ang}|\leq 1/(4n)$ 
		\begin{equation}
		\label{eq:application_of_bernstein_inequality}
		|h(\ang)| \ge |h(\tilde{\ang})| - |h(\tilde{\ang}) - h(\ang)| \ge \lVert h \rVert_\infty - |\ang- \tilde{\ang}| \lVert h^\prime \rVert_\infty \ge \frac{1}{2} \lVert h \rVert_\infty.
		\end{equation} 
		Furthermore, since the $\{\xi_k\}$ are assumed to be sub-Gaussian and independent, there exists some $C>0$ so that
		\begin{equation}
			\label{eq:sub_gaussian_inequality}
			\bE \left[e^{s h(\ang)}\right] = \prod_{k=1}^{n} \bE\left[e^{s n^{-3/2} \xi_k k \left(1+\frac{A^\prime}{n}\right)^{k-1} \cos\left((k-1)\ang\right)}\right] \leq e^{C s^2 n^{-3} \sum_{k=1}^n k^2}  \leq e^{Cs^2}
		\end{equation}
		for all $s\in \bR$.

		Applying Markov's inequality with \eqref{eq:application_of_bernstein_inequality} and \eqref{eq:sub_gaussian_inequality} gives
		\begin{align*}
		\bP\left(\lVert h \rVert_\infty \ge \la  \right) & \leq e^{-s \la/2} \bE\left[e^{s \lVert h \rVert_\infty/2}\right] \\ & \leq e^{-s\la/2}\bE\left[2n\int_{|\ang-\tilde{\ang}|<\frac{1}{4n}}\left(e^{sh(\ang)} + e^{-sh(\ang)}\right) \dd \ang \right] \\ & \ll n e^{-\la s/2} \int_{\ang \in \bT} \bE (e^{sh(\ang)} + e^{-sh(\ang)}) \dd\ang\\
		&\le n e^{-\la s/2} e^{Cs^2}.
		\end{align*}
		Choose $s=\la/4C$ and note that $\la \ge A \sqrt{\log n}$ we obtain as desired.
		
		Finally, for the second derivatives $\max_{||z|- 1| \le A'/n} |f''(z)|$, after applying the maximum principle it suffices to focus only on  the two circles, 
		\begin{equation*}
				\Big\{|z| = 1+\frac{A^\prime}{n} \Big\} \quad \text{and} \quad \Big\{|z| = 1-\frac{A^\prime}{n} \Big\}\, ,
		\end{equation*}
		over which the real and imaginary parts are trigonometric polynomials, and hence we can use Bernstein inequality again, the details are left for the reader.
\end{proof}

For convenience, denote by
	\begin{equation*}
	\cG:= \left\{ \max_{||z|-1|=O(1/n)} |f^{(k)}(z)|, |\frac{\partial^k f}{ \partial^l \theta \partial^{k-l} \rho} |\le  n^{k+1/2} \log^2n  \,,\ k=0,1,2;\, 0\le l \le k\right\}.
	\end{equation*}

		\begin{cor}
		\label{cor:derivatives}
		We have 
		$$\bP\left(\cG^c\right) \le \exp\left(-\Theta(\log^4 n)\right).$$
	\end{cor}

\subsection{Control on covariances}
\label{subsec-control_covariance}
Here we gather some results on the joint distribution of $X,Y$ and their derivatives at a fixed number of points. 
We begin with the distribution at a single point. 
In the sequel we denote the matrix
		\begin{equation}
			\label{eq:def_of_sigma}
			\Sigma_0 := \begin{pmatrix}
			\frac{1}{2} & 0 & 0 & \frac{1}{4} \\ 0 & \frac{1}{2} & -\frac{1}{4} & 0 \\ 0 & -\frac{1}{4} & \frac{1}{6} & 0 \\ \frac{1}{4} & 0 & 0 & \frac{1}{6}
			\end{pmatrix}.
		\end{equation}
	
	\begin{lemma}
		\label{lem:Sigma0} 
For any fixed $\eps>0$ and $t\in [n^{-1+\eps},\pi - n^{-1+\eps}]$ the (centered) random vector 
		\begin{equation}\label{def:Wang}
			\walk(t) := \left(X(t),Y(t),\frac{1}{n}X^\prime(t),\frac{1}{n}Y^\prime(t)\right)
		\end{equation}
		has  covariance matrix 
		\[
		\Sigma(t) = \bE \walk(t)^\tran \walk(t) =  (
	\text{Id}+O(n^{-\eps})) \Sigma_0\,.
		\]
		where the error term is a matrix with entries of size $O(n^{-\eps})$.
	\end{lemma}

	\begin{proof}
		By applying simple trigonometric identities, we see that 
		\begin{align*}
		\Sigma(t)=
		\frac{1}{n}\sum_{k=0}^{n}\begin{pmatrix}
		\cos^2(kt) & \frac{1}{2} \sin(2kt) & -\frac{k}{2n}\sin(2kt) & \frac{k}{n}\cos^2(kt) \\ \frac{1}{2} \sin(2kt) & \sin^2(kt) & - \frac{k}{n} \sin^2(kt) & \frac{k}{2n} \sin(2kt) \\ -\frac{k}{2n} \sin(2kt) & - \frac{k}{n} \sin^2(kt) & \frac{k^2}{n^2}\cos^2(kt) & -\frac{k^2}{2n^2} \sin(2kt) \\ \frac{k}{n} \cos^2(kt) & \frac{k}{2n}\sin(2kt) & -\frac{k^2}{2n^2} \sin(2kt) & \frac{k^2}{n^2} \sin^2(kt)
		\end{pmatrix}.
		\end{align*} 
		It remains to note that (see for instance \cite[Appendix B]{BCP})
		\[
		\left|\sum_{k=0}^{n} k^{a} \sin(kt)\right| , \left|\sum_{k=0}^{n} k^{a} \cos(kt)\right| = O( n^{a+1-\eps})
		\]
		for $a=0,1,2$ and $t\in [n^{-1+\eps},\pi - n^{-1+\eps}]$, 
		giving $\Sigma(t)=\Sigma_0+O(n^{-\eps})$. Since $\Sigma_0$ is invertible we can factor it out of the additive error,
		and the claim follows.
	\end{proof}

Now for $\bs t=(t_1,\dots, t_k)\in \R^k$ we denote the random vector
\begin{equation}	\label{def:Snt}
\walk(\bs t)  = \big( \walk(t_1),\dots, \walk(t_k)\big)\;\in\, \R^{4k}.
\end{equation}
with covariance matrix
\[
\Sigma(\bs t) = \bE \walk(\bs t)^\tran \walk(\bs t).
\]

Recall Definition \ref{def:spread} on spread points. When the points $t_1,\dots, t_k$ are $\omega(1)$-spread it is easily seen that the covariance matrix decouples into blocks, as shown in the following:

\begin{lemma}[Decorrelation for Gaussian field]
\label{lem:decorr}
Fix $t_1,\dots, t_k\in [0,\pi]$ and non-negative measurable functions $\phi_1,\dots, \phi_k: \R^4\to \R_+$ supported in $B(0,n^{a_1})$. Assume $\bs t=(t_1,\dots, t_k)\in \R^k$ is $n^{a_2}$-spread. 
Let $W_0\in \R^4$ be a centered Gaussian vector with covariance $\Sigma_0$.
Then
\[
\EGauss \prod_{i=1}^k \phi_i( \walk(t_i) ) = (1+O(n^{2a_1-a_2}))\prod_{i=1}^k \bE \phi_i(W_0) . 
\]
\end{lemma}

\begin{proof}	
	We have
		\begin{equation}
			\label{eq:lambda_is_block_plus_error}
			\Sigma(\bs t) = \begin{pmatrix}
			\Sigma_0 & 0 & \dots & 0 \\ 0 & \Sigma_0 & \dots & 0 \\ \vdots & \vdots & \ddots & \vdots \\ 0 & \dots & 0 & \Sigma_0 
			\end{pmatrix} + E
		\end{equation}
		where $\Sigma_0$ is given by \eqref{eq:def_of_sigma}, and all entries in the matrix $E$ are $O(n^{-a_2})$. 
		It is evident from \eqref{eq:lambda_is_block_plus_error} that $\det(\Sigma(\bs t)) = \det(\Sigma_0)^{k} \left(1 + O_k(n^{-a_2})\right)$ and that
		\[
		\Sigma(\bs t)^{-1} = \begin{pmatrix}
		\Sigma_0^{-1} & 0 & \dots & 0 \\ 0 & \Sigma_0^{-1} & \dots & 0 \\ \vdots & \vdots & \ddots & \vdots \\ 0 & \dots & 0 & \Sigma_0^{-1} 
		\end{pmatrix} + \wt{E}
		\]
		where again all entries of $\wt{E}$ are $O_k(n^{-a_2})$.
		Since $\det(\Sigma(\bs t)) = \det(\Sigma_0)^{k} \left(1 + O_k(n^{-a_2})\right)$ we see that
		\begin{align}
			\nonumber
			 \frac{\left||\det\Sigma_0|^{-k/2} - |\det\Sigma(\bs t)|^{-1/2} \right|}{(2\pi)^k}
			 &\int_{(\bR^4)^k} e^{-\frac{1}{2} w^T (\Sigma(\bs t)^{-1} - \wt{E}) w}  \prod_{i=1}^k\phi_i(w_i)\, \dd m(w) \\ 
			 & \ll_k n^{-a_2}\prod_{i=1}^k \bE\phi_i(W_0) \,.
		\end{align}
		Since $\phi_1,\dots,\phi_k$ are supported on $B(0,n^a)$, using that $1-e^{-x}  \ll  x$ for $x$ small we see that
		\begin{align}
			\nonumber
			\int_{(\bR^4)^k} &\left|e^{-\frac{1}{2} w^T \Sigma(\bs t)^{-1} w} -e^{-\frac{1}{2} w^T (\Sigma(\bs t)^{-1} - \wt{E}) w}\right| \prod_{i=1}^k\phi_i(w_i) \, \dd m(w) \\ 
			\nonumber & \ll \int_{(\bR^4)^k}e^{-\frac{1}{2} w^T (\Sigma(\bs t)^{-1} - \wt{E})  w} \left|w^{T}\wt{E} w\right|\prod_{i=1}^k\phi_i(w_i) \, \dd m(w) \\ 
			\label{eq:comparison_inequality_gaussian_are_close} 
			& \ll_k n^{2a_1-a_2}\prod_{i=1}^k \int_{\R^4}e^{-\frac{1}{2} w_i^T \Sigma_0^{-1} w_i} \phi_i(w_i) \, \dd m(w_i) \,.
		\end{align}
		Combining these bounds, we have
		\begin{align*}
			\bE_{\BN(0,1)}&\prod_{i=1}^k \phi_i(W(t_i)) - \prod_{i=1}^k \bE\phi_i(W_0) \\ 
			&= \frac{1}{(2\pi)^{2k}\sqrt{|\det\Sigma(\bs t)|}} \int_{(\bR^4)^k} \exp\left(-\frac{1}{2} w^T \Sigma(\bs t)^{-1} w\right)  \prod_{i=1}^k \phi_i(w_i) \, \dd m(w) \\ 
			& \qquad - \frac{1}{(2\pi)^{2k}\sqrt{|\det\Sigma_0|^{k}}} \int_{(\bR^4)^k} e^{-\frac{1}{2} w^T (\Sigma(\bs t)^{-1} - \wt{E} ) w}  \prod_{i=1}^k\phi_i(w_i)  \, \dd m(w) \\
			&\ll_k n^{2a_1-a_2} \prod_{i=1}^k \bE \phi_i(W_0)
		\end{align*}
		as claimed. 
\end{proof}
At separations $t-t'= O(n^{-1})$ the vectors $\walk(t),\walk(t')$ become correlated. The following shows that $\Sigma(\bs t)$ is still reasonably well conditioned when $\bs t$ is only $\gamma$-spread for $\gamma\ll1$ (fixed or going to zero). 
Recall that for a symmetric matrix $A$ of dimension $m$ we order its eigenvalues $\lam_1(A)\ge \cdots\ge \lam_m(A)$.

\begin{lemma}\cite[Lemma 3.6]{CNg}\label{lem:cov} 
If $\bs t=(t_1,\dots, t_k)\in [0,\pi]^k$ is $\spread$-spread for some $0<\spread\le1$, then 
\[
\lam_{4k}(\Sigma(\bs t)) \gg_k \spread^{6k-3} .
\]
\end{lemma}

\subsection{Small-ball estimates and CLTs}
\label{sec:distribution}

The following result from \cite{CNg} gives a small ball estimate for the distribution of $W(t_1,\dots, t_k)$ at arbitrary polynomial scales.

\begin{theorem}\cite[Theorem 3.4]{CNg}
\label{thm:smallball}
Let $\bs t=(t_1,\dots,t_k)\in [0,\pi]^k$ be $n^\smooth$-smooth and $\spread$-spread for some $\smooth\in (0,1)$ and $\omega(n^{-1/8k})\le \spread<1$. For any $K<\infty$ and any ball $B$ of radius $\delta\ge n^{-K}$, 
\[
\bP \big( \walk(\bs t) \in B \big) = O_{K,\smooth}(\spread^{-3k}\delta^{4k}).
\]
\end{theorem}

While the above estimates are quite strong, we will also need bounds at non-smooth points, especially near the edge. We have the following:

\begin{lemma}\label{lemma:smallball:edge} Assume that $\frac{\la}{n} \le  t \le \pi -\frac{\la}{n} $. Then for any $\delta \gg n^{-1/2}$ we have
$$\bP(|f(e^{i t})|/\sqrt{n} \le \delta)= O(\la^{-2}\delta^2).$$
\end{lemma}
We will present a proof of this result in Appendix \ref{appendix:smallball}.

Next we recall the following result from \cite{CNg} providing a fine-scale comparison of the distribution of $W(\bs t)$ with a Gaussian vector, under the assumption that the points $t_1,\dots, t_k$ are both smooth and spread.  	

\begin{theorem}\cite[Theorem 3.2]{CNg}
\label{thm:box}
Let $\bs t= (t_1,\dots, t_k)\in[0,\pi]^k$ be $n^\smooth$-smooth and $1$-spread for some $\smooth>0$. Fix $K>0$ and let $Q\subset\R^{4k}$ be a box (direct product of intervals) with side lengths at least $n^{-K}$.
Then
\[
\big| \bP\big(  \walk(\bs t) \in  Q\big) 
- \PGauss\big(  \walk(\bs t) \in Q \big) \big| 
\ll n^{-1/2} |Q|
\]
where $|Q|$ is the volume of $Q$, and the implied constant depends only on $k, \smooth, K$, 
and the sub-Gaussian constant for $\xi$. 
\end{theorem}

The above results let us control the measure, under the law of $\walk(\bs t)$, of domains in $\R^{4k}$ that can be accurately approximated or covered by unions of cubes or balls of (any) polynomially-small size.

\begin{definition}	\label{def:good.domain}
Say a domain $\cD\subset \R^d$ with piecewise smooth boundary is \emph{$(K,L)$-good} if $\partial D$ can be covered by a family of cubes $Q$ with corners in the scaled lattice $n^{-K}\Z^d$, with total volume $\sum|Q|\le n^{-L}$. 
\end{definition}

Note that if $\cD_1,\cD_2$ are $(K,L)$-good, then $\cD_1^c$, $\cD_1\cup \cD_2$ and $\cD_1\cap \cD_2$ are $(K,L-1)$-good (for all $n$ sufficiently large). 
\Cref{thm:smallball} and a covering argument yield the following:

\begin{prop}
\label{prop:domain.UB}
Let $\bs t\in [0,\pi]^{k}$ be as in \Cref{thm:smallball} and let $\cD\subset \R^{4k}$ be $(K,L)$-good for some $K,L>0$. 
Then
\[
\bP( \walk(\bs t) \in \cD) \ll_{K,\smooth} \gamma^{-3k}( m(\cD) + n^{-L}).
\]
In particular, for $k=1$ and $t=t_1$ that is $n^\kappa$-smooth, 
\[
\bP(\walk(t) \in \cD) \ll_{K,\smooth} m(\cD) + n^{-L}.
\]
\end{prop}

Combining \Cref{thm:box} and \Cref{lem:decorr}, we obtain that the joint law of $(W(t_1),\dots, W(t_k))$ in phase space approximately factorizes into independent Gaussian measures as soon as the times $t_1,\dots, t_k$ are sufficiently spread.  

	\begin{prop}
	\label{prop:compare}
	Fix $n^\smooth$-smooth points $t_1,\dots, t_k\in [0,\pi]$, and let $W_0\in \R^4$ be a centered Gaussian with covariance matrix $\Sigma_0$ as in \eqref{eq:def_of_sigma}.
\begin{enumerate}
\item[(a)] If $t_1,\dots, t_k$ are 1-spread, then for any $(K,L)$-good domain $\cD\subset \R^{4k}$, 
	\begin{equation}
	\bigg| \bP \big(W(\bs t) \in \cD\big)-  \PGauss\big(W(\bs t) \in \cD\big) \bigg| \ll n^{-1/2}m(\cD) + n^{-L} \,.
	\end{equation}
\item[(b)] For any $a_1,a_2>0$, if $t_1,\dots, t_k$ are $n^{a_1}$-spread, then for any $(K,L)$-good domains $\cD_1,\dots, \cD_k\subset B(0,n^{a_2})$ in $\R^4$,
	\begin{equation}
	\bigg| \bP \big(W(t_1)\in\cD_1,\dots,W(t_k)\in \cD_k\big)-  \prod_{j=1}^k \bP(W_0\in \cD_j) \bigg| \ll n^{-\min(\frac12,a_1-2a_2)}\prod_{j=1}^k m(\cD_j) + n^{-L}  \,.
	\end{equation}
\end{enumerate}
Here the implied constants depend on $K,L, \kappa,k$, and the sub-Gaussian moment of $\xi$.
	\end{prop}
	
	\begin{remark}
	By a routine approximation argument, the above result is equivalent  with the statement that for any nonnegative observables $\phi_1,\dots,\phi_k:\R^4\to \R$ supported on $B(0,n^{a_1})$ with first-order partial derivatives uniformly bounded by $n^{K}$, 
		\begin{equation}
	\bigg| \bE \prod_{j=1}^k \phi_j(W(t_j)) -  \prod_{j=1}^k \bE\phi_j(W_0) \bigg| \ll n^{-\min(\frac12,a_2-2a_1)}\prod_{j=1}^k \int_{\R^4} \phi_j \dd \mLeb +n^{-L} \,,
	\end{equation}
	(up to modification of parameters $K,L$). 
	\end{remark}

\section{Proof of \Cref{prop:exceptional}}\label{sec:exceptional}
Recall that the measure of $\badarcs$ on the unit circle is bounded by $O(n^{2\kappa-1})$. Both \eqref{badarcs.i} and \eqref{badarcs.ii} will follow once we prove that
\begin{equation}
	\label{eq:minima_on_bad_arcs_not_small}
	\bP\left(\min_{t\in \badarcs} |f(t)| \leq \log^4n\right) \ll n^{-\kappa/4}
\end{equation}
for sufficiently small $\kappa>0$. Indeed, denote by $\cB:= \{\min_{t\in \badarcs} |f(t)| \leq \log^4n\}$, and observe that if there exist $\theta_\al\in \badarcs$ such that $\cA_\al$ hold, then we must have
\begin{equation*}
	\sqrt{|X(\theta_\al)|^2+|Y(\theta_\al)|^2} \leq 2n^{-2/3} \ll \frac{\log^4n}{\sqrt{n}}.
\end{equation*}
This implies the inclusion $\left\{\mu^\flat(\bR)>0 \right\} \subset \cB$ and hence \eqref{badarcs.i}. For \eqref{badarcs.ii}, observe that on the event
\begin{equation*}
	\{\exists \ang \in \badarcs, z\in Z(f) : |z - e^{i \ang}| \le n^{-3/2}\} \cap \cG
\end{equation*}
we can Taylor expand around the root $z\in Z(f)$ and get that
\[
|f(e^{i\theta})| \leq |f(z)| + n^{3/2}\log^2n|z-e^{i\theta}| \ll \log^2n.
\] 
Therefore, for $n$ large enough, we get that
\begin{equation*}
	\bP\left(\exists \ang \in \badarcs, z\in Z(f) : |z - e^{i \ang}| \le n^{-3/2}\right) \leq \bP\left(\cB\cap \cG \right) + \bP\left(\cG^c\right) \ll n^{-\kappa/4}
\end{equation*}
where the last inequality follows from \eqref{eq:minima_on_bad_arcs_not_small} and Corollary \ref{cor:derivatives}.

By the above reasoning, the proof of \Cref{prop:exceptional} will follow once we prove \eqref{eq:minima_on_bad_arcs_not_small}.
\begin{proof}[Proof of \eqref{eq:minima_on_bad_arcs_not_small}]
	Note that by Corollary \ref{cor:derivatives} we may always assume that $\cG$ holds. We first show that $|f|$ cannot be too small on the set $\badarcs \cap [\frac{1}{n^{1+\kappa}}, \pi - \frac{1}{n^{1+\kappa}}]$. Indeed, we cover this set by non-overlapping interval $\{J_\beta\}_{\beta=1}^{B}$ with $|J_\beta| \leq n^{-7/4}$ and $B\ll n^{7/4}/n^{1-2\kappa} = n^{3/4+2\kappa}$. Denote the mid-point of each $J_\beta$ by $x_\beta$. By Taylor expansion, we have
	\[
	|f(e^{ix_\beta})| \leq |f(z)| + n^{3/2}\log^2n|z-e^{ix_\beta}| \ll |f(z)| + \frac{\log^2n}{n^{1/4}}
 	\] 
 	for all $z\in J_\beta$. Therefore, the event $\{\min_{J_\beta}|f|\leq \log^4n\}\cap \mathcal{G}$ is contained in $\{|f(x_\beta)|\leq \log^5n\}$. By applying the union bound together with Lemma \ref{lemma:smallball:edge} (with $\lambda=n^{-\kappa}$) we see that
 	\begin{align}
 		\nonumber \bP\left(\min_{t\in \badarcs \cap [\frac{1}{n^{1+\kappa}}, \pi - \frac{1}{n^{1+\kappa}}]} |f(t)|\leq \log^4 n\right) &\ll \sum_{\beta=1}^{B} \bP\left(|f(e^{ix_\beta})|\leq \log^5n \right) + \bP(\cG^c) \\ \label{eq:first_union_bound_on_bad_arcs} &\ll B \frac{n^{2\kappa+o(1)}}{n} \ll n^{4\kappa + o(1) -1/4}.
 	\end{align}
 	For the intervals $[0,n^{-1-\kappa}]$ and $[\pi-n^{-1-\kappa},\pi]$ we argue similarly, but instead of \Cref{lemma:smallball:edge} we use the classical Berry-Esseen bound. Dividing both $[0,n^{-1-\kappa}]$ and $[\pi - n^{-1-\kappa},\pi]$ into $O(n^{1/2-\kappa/2})$ intervals $I_j$ of length $n^{-3/2-\kappa/2}$. Under $\cG$, if $|f(e^{i\theta})|\leq \log^4n$ for some $\theta\in I_j$, then $|f(e^{i\theta_j})|\leq \log^6n$, where $\theta_j$ is the mid point of $I_j$. Next, by the Berry--Esseen theorem \cite[Theorem~12.4]{bha_rao}, 
 	$$\bP(|f(e^{i\ang_j})|<\log^6n)\ll \frac{\log^6n}{\sqrt{n}}.$$
 	A union bound gives
 	\[
 	\bP(\exists \ang \in \cup I_j \,:\,|f(e^{i\ang})|<\log^4n) \ll n^{1/2-\kappa/2} \frac{\log^6n}{\sqrt{n}} \ll n^{-\kappa/4}.
 	\]
 	Together with \eqref{eq:first_union_bound_on_bad_arcs} we get that $\bP(\cB)\ll n^{-\kappa/4}$, which is what we wanted.
 \end{proof}

\section{Proof of \Cref{prop:moments}}\label{sect:moments}
For intervals $U,V\subset \bR$ and $r\ge0$ we denote the domain
\begin{align}\label{def:domain}
\cD_{U,V,r} &= 
\bigg\{ (x,y,x',y')\in \R^4: 
\frac{yx^\prime-xy^\prime}{x'^2+y'^2}\in\frac1nU\,,\ \frac{xx' + yy'}{x'^2 + y'^2} \in \frac{n}NV \,, \   x'^2+y'^2 < r^2 \bigg\}	\notag\\
&= \bigg\{ (w,z)\in \R^2\times \R^2: 
\frac{w\cdot z^\perp}{|z|^2}\in \frac1nU\,,\ \frac{w\cdot z}{|z|^2} \in \frac{n}NV,\  |z| < r \bigg\}
\end{align}
where for $z= (x',y')$ we denote $z^\perp = (-y',x')$. We have
\begin{equation}
\Big\{ n^2\rho_\al \in  U, \ N(\theta_\al-\dang_\al) \in V \,,\,   |X'(\theta_\al)|^2+|Y'(\theta_\al)|^2 < r^2n^2
\,\Big\}
= \Big\{ W(\theta_\al) \in \cD_{U,V,r} \Big\}. 
\end{equation}
One further sees that the events $\cA_\al, \cA_\al',\cA_\al''$ from \eqref{eq:defA}, as well as $\{Z_\al\in [a,b]\}$ for any interval $[a,b]$, can all be expressed as events that $\walk(\theta_\al)$ lies in a domain of the form
\begin{equation}	\label{domains}
(\cD_{U,V,R}\setminus \cD_{U,V,r}) \cap I\times J\times I'\times J'
\end{equation}
for some (possibly infinite) intervals $U,V,I,I',J,J'\subseteq \R$ and $0\le r<R\le \infty$. 
Recall the notion of a $(K,L)$-good domain from \Cref{sec:distribution}. 
One easily sees the following:

\begin{fact}	\label{fact:domains}
For any $A, L>0$, $r\ge0$ and intervals $U,V$ such that $r,|U|,|V|\le n^A$, and (possibly infinite) intervals $I,J,I',J'\subseteq\R$, the domain \eqref{domains}
 is $(K,L)$-good for some $K(A,L)$ sufficiently large. 
\end{fact}

Indeed, the cross-section of $\cD$ obtained by fixing $z$ is a rectangle in $\R^2$ of area $|U||V||z|^2/N$ with corners that vary smoothly (with polynomially bounded derivatives) in $z$. 
Integrating this expression over $z$ we find
\begin{equation}	\label{mLebD}
m(\cD_{U,V,r}) = \frac{|U||V|}{N} \int_{|z|\le r} |z|^2 \dd m(z) = \frac{\pi r^4 |U||V|}{2N}.
\end{equation}
	Modulo \Cref{prop:compare}, the proof of \Cref{prop:moments} essentially comes down to a computation of the measure of $\cD_{[a,b],[-\frac\pi2,\frac\pi2],\infty}$ under the Gaussian law of $\walk(\theta_\al)$. 
	In fact, in order to control some bad events we will need the following more general result allowing finite $r=r(n)$ (in particular allowing $r=o(1)$).

	\begin{lemma}
	\label{lem:gaussian.computation}
			Fix an arbitrary $\eps>0$ and let $t\in [n^{-1+\eps},  \pi -n^{-1+\eps}]$. 
For any $r>0$ (possibly infinite or depending on $n$) and intervals 
$U,V\subset[-n^{0.1},n^{0.1}]$,
\begin{align}
&\bP_{\BN(0,1)}\Big( W(t) \in \cD_{U,V,r}
\Big) 
 = \bigg(\frac{12}{\pi^2}+
 O(n^{-\eps/2})
 \bigg)\frac{|U||V|}{N} \cdot \int_{|z|\le r}|z|^2 e^{-12|z|^2} \dd m(z)  \,.	\label{gauss.comp1}
\end{align}
In particular, taking $V=[-\frac\pi2,\frac\pi2]$ and $r=\infty$, we have that for
$\alpha$ so that $\theta_\alpha\in [n^{-1+\eps}, \pi-n^{-1+\eps}]$,
\begin{equation}	\label{gauss.comp3}
\bP_{\BN(0,1)}\Big(n^2\rho_\al \in U, \ \dang_\al\in I_\al \Big)
= \bigg(\frac1{12}+
O(n^{-\eps/2})
\bigg)\frac{|U|}{ N} \,.
\end{equation}

\end{lemma}

	\begin{proof}[Proof of \Cref{lem:gaussian.computation}] 
 	By the monotone convergence theorem we may assume $r$ is finite (possibly, $n$ dependent).
	Recall from \Cref{lem:Sigma0} that $(n^{-1}X'(t), n^{-1}Y'(t))$ is a centered Gaussian with covariance 
	$\frac16\text{Id} + O(n^{-\eps})$,
	and for $z
	\in \bR^2$, 
	the conditional distribution of $\left(X(t),Y(t)\right)$ given 
	$ (n^{-1}X^\prime(t) ,n^{-1}Y^\prime(t))=z$
	is  Gaussian with mean 
	$-\frac{3}{2}z^\perp + O(n^{-\eps})|z|$
	and covariance matrix 
	$\frac{1}{8}\text{Id} + O(n^{-\eps})$, where we write $O(n^{-\eps})$ for a matrix or vector with entries of size $O(n^{-\eps})$ (with implicit constants independent of $z$).
	Denoting by $\bP_{z}$ 
	the
	conditional distribution of $\left(X(t),Y(t)\right)$ given that $ (n^{-1}X^\prime(t),n^{-1}Y^\prime(t)) = z$, we can express the left hand side of \eqref{gauss.comp1} as
		\begin{align}
		\label{eq:gaussian_computation_first_intensity}
&\quad\bP_{\BN(0,1)}\left(\frac{YX^\prime/n-XY^\prime/n}{(X^\prime/n)^2+(Y^\prime/n)^2}\in\frac1nU,\ \frac{XX^\prime/n + YY^\prime/n}{(X^\prime/n)^2+(Y^\prime/n)^2} \in \frac{n}NV,\  |X'|^2+|Y'|^2 \le r^2n^2\right) \nonumber \\
&\quad =\frac{3}{\pi}\int_{|z|\le r}e^{-(3+O(n^{-\eps}))|z|^2} \bP_{z}\left(\frac{\avg{(X,Y),z^\perp}}{|z|^2}\in \frac1nU, \ \frac{\avg{(X,Y),z}}{|z|^2}\in \frac{n}NV  \right) \dd m(z)\,.	
		\end{align}
Under $\bP_z$ we can express 
\[
(X,Y) = -\frac32z^\perp + O(n^{-\eps})|z| + (\frac1{\sqrt{8}}\text{Id} + O(n^{-\eps}))G
\]
for a standard Gaussian $G=(G_1,G_2)\in \R^2$, and thus by the rotational invariance of $G$ we may re-express 
\begin{align}	\label{gauss.comp4}
&\bP_{z}\left(\frac{\avg{(X,Y),z^\perp}}{|z|^2}\in \frac1nU, \ \frac{\avg{(X,Y),z}}{|z|^2}\in \frac{n}NV  \right)
= \bP\left( \Big(\frac1{\sqrt{8}}\text{Id} + O(n^{-\eps})\Big)G \in |z| (I\times J) \right)
\end{align}
for the shifted intervals $I =  \frac32 + O(n^{-\eps}) + \frac1nU$ and $J= O(n^{-\eps})+\frac n{N} V$.
By our assumptions on $U,V$, the Gaussian density varies by a factor of at most 
\[
1+ O(|z|^2\max(\frac1n|U|, \frac{n}{N}|V|)) =1+O( |z|^2 n^{-1/2}), 
\]
say, on $|z|(I\times J)$, and the latter probability in \eqref{gauss.comp4} is thus
\[
(1+O(|z|^2 n^{-1/2})) \frac{4}{\pi} |z|^2\frac{|U||V|}{N}e^{ - (9 + O(n^{-\eps}))|z|^2}.
\]
Combining with \eqref{eq:gaussian_computation_first_intensity}, we have shown that the left hand side of \eqref{gauss.comp1} is
\begin{equation}
(1+O(n^{-1/4}))\frac{|U||V|}{N} \frac{12}{\pi^2} \int_{|z|\le r}e^{-(12+O(n^{-\eps}))|z|^2}|z|^2 \dd m(z)
\end{equation}
uniformly for $r\le n^{1/8}$. 
Now estimating 
\begin{align*}
&\bigg| \int_{|z|\le r}e^{-12|z|^2}|z|^2 \dd m(z) 
 - \int_{|z|\le r}e^{-(12+O(n^{-\eps}))|z|^2}|z|^2 \dd m(z) \bigg|\\
 &\qquad\qquad\le \int_{|z|\le r} e^{-12|z|^2 }| 1- e^{O(n^{-\eps})|z|^2} | |z|^2\dd m(z) \\
 &\qquad \qquad\ll 
 n^{-\eps/2} \int_{|z|\le r} e^{-12|z|^2} |z|^2\dd m(z) 
\end{align*}
uniformly for $r\le n^{\eps/4}$, we obtain \eqref{gauss.comp1} for this range.

For larger $r$ we bound the right hand side of \eqref{gauss.comp4} by $O(|z|^2|U||V|/N)$, and so the contribution to \eqref{eq:gaussian_computation_first_intensity} from integration over $n^{\eps/4}\le |z|\le r$ is
\[
\ll \frac{|U||V|}{N}  \int_{|z|\ge n^{\eps/4}} |z|^2 e^{-(3+O(n^{-\eps}))|z|^2} \dd m(z)
\ll \frac{|U||V|}{N} e^{-n^{\eps/4}}
\]
uniformly in $r\ge n^{\eps/4}$.
This is easily absorbed by the error term in \eqref{gauss.comp1}, giving the claim.
	\end{proof}
	
	\begin{proof}[Proof of \Cref{prop:moments}]
		We first assume that $U=[a,b]$ is a compact interval. Fix a $k$-tuple of indices $(\al_1,\dots,\al_k)\in [N]^k$ as in the statement of \Cref{prop:moments}. 
		Recall that the events $\{Z_{\al_i}\in [a,b]\}$ are of the form $\{W(\theta_{\al_i})\in \cD\}$ for a domain $\cD\subset B(0,10\log^2n)$ in $\R^4$ that is $(K,10)$-smooth for some $K>0$ sufficiently large (see \Cref{fact:domains}). 
		Thus, the claim will follow from an application of \Cref{prop:compare} once we show
		\begin{equation}	\label{moments:goal1}
		\bP_{\BN(0,1)}\left(Z_\al\in [a,b] \right) = \frac{1}{12}\cdot\frac{b-a}{N} + o\left(\frac{1}{N}\right).
		\end{equation}
		for each $\al=\al_1,\dots, \al_k$. Since
		\[
		\Big\{ Z_\al\in [a,b]\Big\} = \Big\{ n^2\rho_\al \in  [a,b],\ \dang_\al \in I_\al ,\ \cA_\al^{\dprime} \Big\}\,,
		\]
		from \eqref{gauss.comp3} and \Cref{cor:derivatives} we see that it suffices to show
		\begin{equation}	\label{moments:goal2}
		\bP_{\BN(0,1)}\left(Z_\al\in [a,b] \,,\, (\cA_\al'')^c\,,\, \cG \right)  = o(N^{-1}).
		\end{equation}
		On $\{Z_\al\in [a,b],\,\cG\}$,  the projection of $(X(\theta_\al), Y(\theta_\al))$ in the directions $(X'(\theta_\al), Y'(\theta_\al))$ and $(X'(\theta_\al), Y'(\theta_\al))^\perp$ are of size $O(n^{-1}\log^2n)$ and $O(N^{-1}n\log^2n) = n^{-1+o(1)}$, respectively, so 
		\[
		X(\theta_\al)^2+Y(\theta_\al)^2 \le n^{-2+o(1)} = o(n^{-4/3}).
		\]
		Thus, on $\{Z_\al\in [a,b],\,(\cA_\al'')^c,\,\cG\}$ we must have that either $\{|X'(\theta_\al)| \le n\log^{-2K_0} n\}$ or $\{|Y'(\theta_\al)| \le n\log^{-2K_0} n\}$. 
		Hence,
		\[
		\Big\{Z_\al\in [a,b],\,(\cA_\al'')^c,\,\cG\Big\}\subset \Big\{ W(\theta_\al) \in \cD_{U, [-\frac\pi2,\frac\pi2], 2\log^{-2K_0} n}\Big\}.
		\]
		From \eqref{gauss.comp1} the latter event has probability 
		\[
		\ll \frac{1}{N} \int_{|z|\le \log^{-2K_0} n} |z|^2 e^{-12|z|^2}\dd m(z) = o(N^{-1})\,,
		\]
		which yields \eqref{moments:goal2} and hence \eqref{moments:goal1}. Moving to consider $U\subset\bR$ which is a finite union of compact intervals, we note that 
		\[
		\bP_{\BN(0,1)}\left(Z_\al\in U \right) = \frac{1}{12}\cdot\frac{|U|}{N} + o\left(\frac{1}{N}\right)
		\]
		follows from \eqref{moments:goal1} by finite additivity. Combining the above display with \Cref{prop:compare} we complete the proof of \Cref{prop:moments}. 
	\end{proof}

\section{Proof of Proposition \ref{prop:PPequivalence}}	\label{sect:PPequivalence}

	For an interval $U\subset \bR$, denote the respective annular domain by
	\begin{equation*}
		U_\al =\left\{z\in \bC \mid n^2\left(1-|z|\right) \in U,\ \text{arg}(z)\in I_\al \right\}.
	\end{equation*}
	Recall the definition $F_\al$ from \eqref{eq:def_of_linear_approximation}.
	\begin{claim}
		\label{claim:symmetric_difference_is_small}
		Fix an arbitrary compact interval $U$. There exists $\delta \ll n^{-5/2}$ such that for any $\al\in [N]$, on the event $\cA_\al''\cap \cG$, 
\[
U_\al^- \subset F_\al^{-1}(\frac1{\sqrt{n}} f(U_\al)) \subset U_\al^+
\]
where $U_\al^+$ is the $\delta$-neighborhood of $U_\al$, and $U_\al^-$ is the complement of the $\delta$-neighborhood of $U_\al^c$.
	\end{claim}

The key point is that $U_\al^+\setminus U_\al^-$ is small and $(K,L)$-good, and can hence be controlled using \Cref{prop:domain.UB}.
	
	\begin{proof}
		We use Taylor expansion to bound the error term in \eqref{eq:system_of_two_equation_for_root}. 
		By the restriction to $\CG$ we have
		\begin{equation}
		\label{eq:linear_approxiamtion_for_f_order_2}
		\left|\frac1{\sqrt{n}}f(t,\rho) - F_\al(t,\rho)\right| \ll \frac{n^{2}\log^4 n}{N^{2}} \ll \frac{1}{n^{2-o(1)}}
		\end{equation}
		for all $(t,\rho)$ such that $(1+\rho)e^{it}\in U_\al$. On the event $\cA_\al''$, the affine transformation $F_\al$ is invertible as
		\begin{equation*}
			\left|\det \begin{pmatrix}
			X^\prime(\ang_\al) & Y^\prime(\ang_\al) \\ Y^\prime(\ang_\al) & -X^\prime(\ang_\al)
			\end{pmatrix}\right| = (X^\prime(\ang_\al))^2 + (Y^\prime(\ang_\al))^2 \in \left[n^2 \log^{-4K_0}n ,2n^2\log^4n\right].
		\end{equation*}
		By applying $F_\al^{-1}$ to \eqref{eq:linear_approxiamtion_for_f_order_2} and by observing that
		\begin{equation}\label{eqn:F:nice}
			\|\begin{pmatrix}
			X^\prime(\ang_\al) & Y^\prime(\ang_\al) \\ Y^\prime(\ang_\al) & -X^\prime(\ang_\al)
			\end{pmatrix}^{-1} \|_2= \frac{1}{(X^\prime(\ang_\al))^2 + (Y^\prime(\ang_\al))^2}\| \begin{pmatrix}
			X^\prime(\ang_\al) & Y^\prime(\ang_\al) \\ Y^\prime(\ang_\al) & -X^\prime(\ang_\al)
			\end{pmatrix} \|_2 \ll \frac{\log^{6K_0}n}{n} 
		\end{equation}
		we get that
		\begin{equation*}
			\left|(1+\rho)e^{it} - F_\al^{-1}  (\frac1{\sqrt{n}}f(t,\rho)) \right| \ll \frac{\log^{6K_0}n}{n}
			\left|\frac1{\sqrt{n}}f(t,\rho) - F_\al(t,\rho)\right| \ll \frac{1}{n^{3-o(1)}} \,.
		\end{equation*}
		Assuming $n$ is sufficiently large this gives the desired result.
	\end{proof}

	\begin{claim}
		\label{claim:no_two_roots_inside_annular_domain}
We have
		\[
		\bP\left(\bigcup_{\al: \theta_\al\notin \badarcs}  \left\{|Z(f) \cap C_\al|\geq 2 \right\} \right) = o(1).
		\]
	\end{claim}
	\begin{proof}
		Assume that there exist $z_1,z_2\in C_\al$ such that $z_1\not=z_2$ and $f(z_1)=f(z_2)=0$. By the mean value theorem there exist $\zeta_1$ and $\zeta_2$ on the line segment connecting $z_1$ and $z_2$ so that
		\begin{align*}
			&\text{Re}(f^\prime(\zeta_1)) = \text{Re}\left(\frac{f(z_1)-f(z_2)}{z_1-z_2}\right)=0 \\ &\text{Im}(f^\prime(\zeta_2)) = \text{Im}\left(\frac{f(z_1)-f(z_2)}{z_1-z_2}\right)=0.
		\end{align*}
		On the event $\cG$, we apply the mean value theorem again and see that
		\[
		|f^\prime(e^{i\ang_\al}) - f^\prime(\zeta_j)| \ll \frac{1}{N} \max_{||z|-1|\leq n^{-3/2}} |f^{\dprime}(z)| \ll n^{1/2+o(1)},
		\]
		for $j=1,2$. With the above bound we can use Taylor expansion and get
		\begin{equation*}
			|f(e^{i\theta_\al})| \ll |f(z_1)| + \frac{1}{N}|f^{\prime}(e^{i\theta_\al})| \ll \frac{1}{n^{3/2-o(1)}}.
		\end{equation*}
Applying \Cref{cor:derivatives} we have
		\begin{align*}
			&\bP\left(\bigcup_{\al: \theta_\al\notin \badarcs} \left\{|Z(f) \cap C_\al|\geq 2 \right\} \right) 
			\leq  \bP\left(\cG^c\right)+\sum_{\al: \theta_\al\notin \badarcs} \bP \left(\left\{|Z(f) \cap C_\al|\geq 2 \right\},\ \cG \right)  \\ 
			&\qquad \qquad \ll e^{-\log^2n} +\sum_{\al: \theta_\al\notin \badarcs} \bP\left(|f(e^{i\theta_\al})| \leq \frac{1}{n^{3/2-o(1)}}\,,\, |f'(e^{i\theta_\al})| \le n \,\right) \,.
			\end{align*}
For each $\al$ in the last sum, we can cover the range of possible outcomes for $(X,Y,X',Y')$ with balls of radius $n^{-2+o(1)}$ and bounded overlap and apply the union bound and 
\Cref{thm:smallball} (recall that points outside of $\badarcs$ are 1-smooth and hence 1-spread) to bound each term by $O(n^{-3})$ (with plenty of room). Summing this bound over the $N=\lf \frac{n^2}{\log^{K_0} n}\rf$ values of $\al$ yields the claim.
	\end{proof}

	Now we conclude the main result of the section.
	\begin{proof}[Proof of Proposition \ref{prop:PPequivalence}] 
		First, we will show that $\bP\left(\mu_f^\sharp(U)>\nu_f^\sharp(U)\right)= o(1)$. Indeed we have the inclusion 
		\[
		\left\{\mu_f^\sharp(U)>\nu_f^\sharp(U)\right\} \subset \bigcup_{\al :\theta_\al\notin \badarcs } \left\{\cA_\al,\ Z_\al\in U,\ Z(f)\cap U_\al=\emptyset \right\}.
		\]
		On the event $Z_\al \in U$ we know that $(1+\rho_\al)e^{i\dang_\al} \in U_\al$, while on the event $\cA_\al$ the affine map $F_\al$ is invertible and we get that
		\[
		0\ \not\in f(U_\al) \implies (1+\rho_\al)e^{i\dang_\al} \not\in F_\al^{-1} \left(\frac1{\sqrt{n}}f(U_\al)\right).
		\]
		Applying \Cref{claim:symmetric_difference_is_small} we get
		\begin{align*}
			\bP\left(\cA_\al,\ Z_\al\in U,\ Z(f)\cap U_\al=\emptyset, \ \cG\right)  &\leq \bP\left(\cA_\al,\ (1+\rho_\al)e^{i\dang_\al}\in U_\al \setminus F_\al^{-1} \left(\frac1{\sqrt{n}}f(U_\al)\right)\,,\,\CG\, \right)\\
			& \leq \bP\left( (1+\rho_\al)e^{i\dang_\al}\in U_\al^+\setminus U_\al^- \,,\ \cG\right)\,.
		\end{align*}
		Recalling from \Cref{claim:symmetric_difference_is_small} that $U_\al^+\setminus U_\al^-$ is the $\delta$-neighborhood of the boundary of $U_\al$, we can cover the corresponding set in the $(\rho,t)$ by four rectangles $I\times J$ of area $O(\delta n^2)$.
Thus, on the last event we have that $W(\theta_\al) \in \cD_{I,J,n^{o(1)}}$ for one of four possibilities for $(I,J)$, each satisfying $|I|\times|J|\le \delta n^2\ll n^{-1/2}$. It follows from \Cref{prop:domain.UB} and \eqref{mLebD} that
\[
\bP\left( (1+\rho_\al)e^{i\dang_\al}\in U_\al^+\setminus U_\al^- \,,\ \cG\right) \ll \frac{1}{n^{1/2-o(1)}N} =o(N^{-1})
\]
Summing over $\al$ we get
		\begin{align*}
			\bP&\left(\mu_f^\sharp(U)>\nu_f^\sharp(U)\right)  \leq o(1)+\sum_{\al:\theta_\al\notin\badarcs } \bP\left(\cA_\al,\ Z_\al\in U,\ Z(f)\cap U_\al=\emptyset, \ \cG\right) = o(1).
		\end{align*}	
		The proof of the proposition will follow once we show that 
		\begin{equation}
			\label{eq:nu_f_bigger_mu_f_has_small_prob}
			\bP\left(\mu_f^\sharp(U)<\nu_f^\sharp(U) \right) = o(1).
		\end{equation}
		By Claim \ref{claim:no_two_roots_inside_annular_domain} and \Cref{cor:derivatives} the left hand side above is bounded by
		\begin{align}
			\label{eq:nu_f_bigger_mu_f_has_small_prob_first_reduction}
			& \bP\left(\bigcup_{\al:\theta_\al\notin\badarcs }\{ \cA_\al^\dprime, \ Z_\al\not\in U,\ |Z(f)\cap U_\al| = 1 \} \right)
			\notag \\ &\quad+ \bP\left(\bigcup_{\al:\theta_\al\notin\badarcs }\{ {\left(\cA_\al^{\dprime}\right)}^{c},\ |Z(f)\cap U_\al| = 1 ,\cG\} \right) + o(1).
		\end{align}
		Similar to the previous argument, on the event $\{\cA_\al'',\ Z_\al\not\in U \}$ the map $F_\al$ is non-degenerate and we have
		\[
		(1+\rho_\al) e^{i\dang_\al} \not\in U_\al,
		\]
		while the assumption that $|Z(f)\cap U_\al| = 1$ implies that $0\in f(U_\al)$. This tells us that
		\[
		(1+\rho_\al) e^{i\dang_\al} \in F_\al^{-1}\left(\frac1{\sqrt{n}}f(U_\al)\right) \setminus U_\al \subset U_\al^+ \setminus U_\al^-\,.
		\]
Arguing with \Cref{claim:symmetric_difference_is_small} as we did for the events $\{Z_\al\in U, Z(f)\cap U_\al=\emptyset\}$, we get that
		\begin{align*}
			 &\bP\left(\bigcup_{\al:\theta_\al\notin\badarcs }\{ \cA_\al'', \ Z_\al\not\in U,\ |Z(f)\cap U_\al| = 1 \} \right) \\ &\le o(1)+ \sum_{\al:\theta_\al\notin\badarcs }\bP\left((1+\rho_\al) e^{i\dang_\al} \in U_\al^+ \setminus U_\al  ,\ \cG \right)  =o(1).
		\end{align*}

		For the second probability in \eqref{eq:nu_f_bigger_mu_f_has_small_prob_first_reduction},
consider $\al\in [N]$ with $\theta_\al\notin \badarcs$ and suppose $\{(\cA_\al'')^c, |Z(f)\cap U_\al|=1,\cG\}$ holds. Let $\xi\in U_\al$ be the root of $f$. By Taylor expansion and the restriction to $\cG$ we have
\[
|f(e^{i\theta_\al})| \le |f(\xi)| + O(n^{3/2}\log^2n) |\xi-e^{i\theta_\al}| \ll \frac{\log^{K_0 + 2}n}{\sqrt{n}}
\]
which implies $|X(\theta_\al)|,|Y(\theta_\al)|\le n^{-1+o(1)}=o(n^{-2/3})$. From this and the restriction to $\cG$, if the event $\cA_\al''$ does not occur then we must have either $|X'(\theta_\al)|< n\log^{-2K_0}n$ or $|Y'(\theta_\al)|<n\log^{-2K_0}n$. Now, from \Cref{thm:box} and the boundedness of the Gaussian density (or \Cref{thm:smallball} and a covering by balls of bounded overlap) we get
\[
\bP\left( |X(\theta_\al)|, |Y(\theta_\al)| \le \frac{\log^{K_0 + 2}n}{n},\, |X'(\theta_\al)|\le n\log^{-2K_0}n \right)
\le \frac{\log^{4}n}{n^2} = o\left(N^{-1}\right),
\]
and similarly
\[
\bP\left( |X(\theta_\al)|, |Y(\theta_\al)| \le \frac{\log^{K_0 + 2}n}{n},\, |Y'(\theta_\al)|\le n\log^{-2K_0}n \right)
= o\left(N^{-1}\right).
\]
Taking the union bound over these two cases, we have thus shown
\[
\bP( (\cA_\al'')^c, |Z(f)\cap U_\al|=1,\cG) =o(N^{-1})
\]		
Applying the above and the union bound over the choices of $\al$, we see that \eqref{eq:nu_f_bigger_mu_f_has_small_prob_first_reduction} is $o(1)$, giving \eqref{eq:nu_f_bigger_mu_f_has_small_prob} as desired. 
	\end{proof}

\section{Proof of Proposition \ref{prop:moments:sum}}\label{sect:moments:sum}  
In this section we argue similarly as in \cite[Section~5]{CNg} with some minor modifications. Expanding the factorial moments $\bE\left[\left(\mu_{f}^\sharp(U)\right)_k\right]$ we see that
\begin{equation*}
	\bE\left[\left(\mu_{f}^\sharp(U)\right)_k\right] = \sum_{\boldsymbol{\alpha} \in E} \bP\left(\cE(\boldsymbol{\alpha})\right),
\end{equation*}
where $E:=\left\{\boldsymbol{\alpha}=(\alpha_1,\ldots,\alpha_k)\in [N]^k \mid \alpha_i\not=\alpha_j \ \forall i\not=j,\, \theta_{\al_i}\not\in \badarcs \right\}$ and
\[
\cE(\al) := \left\{Z_{\al_1}\in U,\ldots, Z_{\al_k} \in U \right\}.
\]
For some sufficiently small $\eps>0$ we consider the set $E^\prime:= \left\{\boldsymbol{\alpha} \in E\mid \left(\theta_{\al_1},\ldots,\theta_{\al_k}\right) \text{is } n^\eps\text{-spread} \right\}$. Since $|E^\prime| = N^k(1+o(1))$, we can use \Cref{prop:moments} and get that
\begin{equation*}
	\sum_{\boldsymbol{\alpha} \in E^\prime} \bP\left(\cE(\boldsymbol{\alpha})\right) = N^{k}\left(\frac{|U|}{12N}\right)^k + o(1) = \left(\frac{|U|}{12}\right)^k + o(1) 
\end{equation*}
so \Cref{prop:moments:sum} will follow once we show that
\begin{equation}
	\label{eq:sum_of_probabilities_over_E_setminus_E_prime}
	\lim_{n\to\infty} \sum_{\boldsymbol{\alpha} \in E\setminus E^\prime} \bP\left(\cE(\boldsymbol{\alpha})\right) = 0.
\end{equation}
The next lemma shows that close roots are typically separated. The proof is a simple modification of \cite[Lemma~2.11]{YZ} or \cite[Lemma~2.2]{CNg} and is deferred to Appendix \ref{appendix:separation}.
\begin{lemma}
	\label{lemma:good_events_are_well_separated}
	On the event $\cG$, for $(\alpha,\alpha^\prime) \in [N]^2$ with $\theta_\al,\theta_{\al^\prime} \not\in \badarcs$ we have
	\begin{enumerate}
		\item If $\cA_\al$ and $\cA_{\al+1}$ occur, then
		\begin{equation*}
			|\dang_\al - \ang_{\al}| \in \left[\frac{\pi}{2N} - \frac{\pi}{2N\log^{K_0}n}, \frac{\pi}{2N}\right]\,.
		\end{equation*}
		\item $\cA_\al$ and $\cA_{\al^\prime}$ cannot occur simultaneously as long as
		\[
		|\theta_\al - \theta_{\al^\prime}| \in \bigg(\frac{\pi}{N},\frac{1}{n\log^{4K_0}n} \bigg]\,.
		\]
	\end{enumerate}
\end{lemma} 

\begin{proof}[Proof of Proposition \ref{prop:moments:sum}]
	By the discussion above, we only need to show that \eqref{eq:sum_of_probabilities_over_E_setminus_E_prime} holds. By Corollary \ref{cor:derivatives} we can assume that $\cG$ holds. The second item in Lemma \ref{lemma:good_events_are_well_separated} implies that we only need to consider tuples $\boldsymbol{\alpha} \in E\setminus E^\prime$ of the form
	\begin{equation}
		\label{eq:rel_tuples_for_small_sum}
		\boldsymbol{\alpha} = \left(\alpha_1,\ldots,\alpha_{k-\ell},\alpha_{1}+1,\ldots, 
		\alpha_\ell
		+ 1\right)
	\end{equation}
	consisting of $\ell$ tuples of the form $(\al_j,\al_j+1)$ for some $0\leq \ell \leq k/2$, while the $k-\ell$ points $\theta_{\alpha_1},\ldots,\theta_{\alpha_{k-\ell}}$ are pairwise separated by at least $n^{-1}\log^{-4K_0}n$ in $[0,\pi]$.
	
	We divide the class of such $\al$ into into two sets $E_1$ and $E_2$, where $E_1$ consists of all $\al\in E\setminus E^\prime$ of the form \eqref{eq:rel_tuples_for_small_sum} (possibly with $\ell=0$) such that $|\theta_{\alpha_i}-\theta_{\alpha_j}|\leq n^{-1+\eps}$ for some $1\leq i < j \leq k-\ell$, and $E_2$ is the set of all $\al \in E\setminus E^\prime$ of the form \eqref{eq:rel_tuples_for_small_sum} with $\ell\ge 1$ and $|\theta_{\alpha_i}-\theta_{\alpha_j}| > n^{-1+\eps}$ for all $1\leq i < j \leq k-\ell$.
	
	For the sum over $E_1$, denote by $E_1^\ell$ the tuples $\al\in E_1$ with $\ell$ neighboring pairs as in \eqref{eq:rel_tuples_for_small_sum}. We have $|E_1^\ell| = O(N^{k-\ell}/n^{1-\eps})$, as there are $O(N/n^{1-\eps})$ choices for the close point with all others fixed. 
	Recalling the notation \eqref{def:domain}, we have
	\[
	\cE(\boldsymbol{\alpha}) \subseteq \big\{\walk(\ang_{\al_i})\in \cD_{U,V,2\log^2n} \;\forall 1\le i\le k\big\}
	\]
	with $V=[-\frac\pi2,\frac\pi2]$.
	Since the points $\theta_{\alpha_1},\ldots,\theta_{\alpha_{k-\ell}}$ are separated by at least $n^{-1}\log^{-4K_0}n$, we can apply Proposition \ref{prop:domain.UB} (recalling \Cref{fact:domains}) along with \eqref{mLebD} with $r=2\log^2n$, to get
	\[
	\sum_{\boldsymbol{\alpha} \in E_1} \bP\left(\cE(\boldsymbol{\alpha})\right) \ll \sum_{\ell=0}^{\lfloor k/2 \rfloor} |E_1^\ell| N^{-k+\ell} \log^{O(k)}n \ll \frac{\log^{O(k)}n}{n^{1-\eps}} = o(1).
	\] 
	
	For the sum over $E_2$, fix a tuple $\boldsymbol{\alpha}\in E_2$ with $\ell$ neighboring pairs with $\ell\ge1$. By the first item in Lemma \ref{lemma:good_events_are_well_separated}, we have the containment of events
	\[
	\{Z_{\al_j}\in U,\, Z_{\al_{j+1}}\in U \} \subset \left\{Z_{\al_j}\in U, \ |\dang_{\al_j} - \ang_{\al_j}| \in \left[\frac{\pi}{2N} - \frac{\pi}{2N\log^{K_0}n}, \frac{\pi}{2N}\right] \right\}\,.
	\]
	The event on the right hand side is contained in the event that
	\[
	W(\theta_\al) \in \cD_{U, V_-, 2\log^2n} \cup \cD_{U,V_+ ,2 \log^2n}
	\]
	with $V_-=[-\frac{\pi}2, -\frac\pi2+\frac{\pi}{2\log^{K_0}n}]$ and $V_+=[\frac\pi2-\frac{\pi}{2\log^{K_0}n}, \frac\pi2]$.
	As above, by \Cref{prop:domain.UB}, \Cref{fact:domains} and \eqref{mLebD} we get that for $\boldsymbol{\alpha}\in E_2$ with $\ell$ neighboring pairs,
	\[
	\bP\left(\cE(\boldsymbol{\alpha})\right) \ll \frac{1}{N^{k-\ell}} \left(\frac{\log^{8}n}{\log^{K_0}n}\right)^{k-\ell}. 
	\]
	Taking $K_0>8$, since the number of such tuples is at most $N^{k-\ell}$, the above display implies that
	\[
	\sum_{\boldsymbol{\alpha}\in E_2} \bP\left(\cE(\boldsymbol{\alpha})\right) = o(1).
	\]
	We have thus proved \eqref{eq:sum_of_probabilities_over_E_setminus_E_prime}, and hence also the proposition.
\end{proof}
\section{Extensions}
\label{sec-8}
The method presented in the paper allows one to consider other point processes related to  $\mu_f$ of 
\eqref{eq:def_mu_f}. Specifically,
consider
the $4$-tuple
\[\widehat Z_\alpha=(
n^2\rho_\al, N(\theta_\al-\tau_\al),
X'(\theta_\alpha)/n, Y'(\theta_\alpha)/n),\]
and introduce the point process
\begin{equation}
  \widehat \mu_f:= \sum_{\al =0}^{N} 
  \delta_{(\theta_\alpha,\widehat Z_\a {\bf 1}_{{\mathcal A}_\a}+(\infty,\infty,\infty,\infty){\bf 1}_{\mathcal{A}_\a^c})}.
	\end{equation}
Note that $\mu_f$ can be obtained from $\widehat \mu_f$ by an appropriate contraction. We then have the following.
\begin{proposition}
\label{prop-conc}
Under the assumptions of Theorem \ref{thm:main},
the process $\widehat \mu_f$ converges to a Poisson point process on 
$[0,\pi]\times \mathbb{R}^4$ with intensity 
$$(12/\pi^2)((x')^2+(y')^2)e^{- 12((x')^2+(y')^2)}
{\bf 1}_{y\in [0,\pi]}d\theta dxdy dx'dy'.$$
\end{proposition}
Indeed, the proof of Proposition \ref{prop-conc}
follows that of Theorem \ref{thm:main}. The main difference is 
in the computation of
$\bP_{\BN(0,1)}(\rho_\al\in U/n^2, \theta_\al-\tau_\al\in V/N, 
X'(\theta_\alpha)\in nI_1, Y'(\theta_\alpha)\in nI_2)$.
Following the lines of the proof of \Cref{lem:gaussian.computation}, for $\theta_\al\in [n^{-1+\eps},\pi- n^{-1+\eps}]$ this probability is given by
\begin{align}
		\label{eq:gaussian_computation_first_intensitybis}
&\quad\bP_{\BN(0,1)}\left(\frac{YX^\prime/n-XY^\prime/n}{(X^\prime/n)^2+(Y^\prime/n)^2}\in\frac1nU,\ \frac{XX^\prime/n + YY^\prime/n}{(X^\prime/n)^2+(Y^\prime/n)^2} \in \frac{n}NV,\  X'\in nI_1, Y'\in nI_2\right)\nonumber \\
&\quad =\frac{3}{\pi}\int_{I_1\times I_2}e^{-(3+
O(n^{-\eps}))|z|^2} \bP_{z}\left(\frac{\avg{(X,Y),z^\perp}}{|z|^2}\in \frac1nU, \ \frac{\avg{(X,Y),z}}{|z|^2}\in \frac{n}NV  \right) \dd m(z) ,
		\end{align}
		and the remainder of the proof proceeds along similar lines, with integration over $z\in B(0, r)$ replaced by integration over $z\in I_1\times I_2$.
		We omit details.

		Proposition \ref{prop-conc} allows one to compute, among other things, the joint law of $n^2 \, {\normalfont \text{dist}}(\mathbb{S}^1,Z(f))$ and
$ n\min_{z\in \mathbb{S}^1} 
|f(z)|$.  

\vskip 1cm
\noindent
\textbf{Acknowledgement} We thank the referee for a careful reading of the 
first version and for useful comments. 
We thank Ziv Huppert for pointing out a mistake in the original version of 
Proposition \ref{prop-conc}.

	\appendix
	\section{Proof of Lemma \ref{lemma:smallball:edge}}  \label{appendix:smallball}
With $a_j = \sin(jt)$ and $b_j=\cos(jt)$ we have $f(e^{it})= \sum_{j=1}^n\xi_j (a_j,b_j)$, viewed as a point in $\R^2$. Set
$$t_0=\delta^{-1}.$$ 
By a standard procedure (see for instance \cite[Eq. 5.4]{AP}) we can bound the small ball probability using the characteristic function,
$$\bP(\frac{1}{\sqrt{n}} f(e^{it}) \in B(w,\delta) ) \le C \Big(\frac{n}{t_0^2}\Big) \int_{\R^{2}} \prod_{j=1}^n |\bE e^{i \xi_j \langle (a_j,b_j),u \rang}|  e^{-\frac{n \|u\|_2^2}{2 t_0^2}} du =: J_1+J_2$$
where in $J_1, J_2$ the integral is restricted to the ranges  $\|u\|_2 \le r_0 =O(1)$ and $r_0 \le \|u\|_2$. Additionally, we can bound
(see for instance \cite[Eq. 9.2]{CNg}),
$$ \prod_{j=1}^n |\bE e^{i \xi_j \langle (a_j,b_j),u \rang}| \le \exp(-c \inf_{c_1\le |a| \le c_2} \sum_j \| a \langle (a_j,b_j),u/2\pi \rang\|_{\R/\Z}^2),$$
where $c_1<c_2$ are positive constants depending on $\xi$.
Thus, if $r_0$ is sufficiently small, then we have $\|a \langle (a_j,b_j),u/2\pi  \rangle\|_{\R/\Z} = |a| \|\langle (a_j,b_j),u/2\pi \rang\|_2$, and so from \Cref{lem:cov} we have that for any unit vector $\Be$, $\sum_j \langle (a_j,b_j), \Be  \rangle^2  \ge c' \lam^2 n$, yielding
$$\sum_j \| a \langle (a_j,b_j),u/2\pi \rang\|_{\R/\Z}^2  \ge c' n\|u\|_2^2\lam^2.$$
Hence
\begin{align*}
J_1 &=  C \Big(\frac{n}{t_0^2}\Big) \int_{\|u\|_2 \le r_0}  \prod_{j=1}^n |\bE e^{i \xi_j \langle (a_j,b_j),u \rang}|  e^{-\frac{n \|u\|_2^2}{2 t_0^2}} du  \le  C  \Big(\frac{n}{t_0^2}\Big) \int_{\|u\|_2 \le r_0}  e^{-\frac{n \|u\|_2^2}{2 t_0^2} - c'n \|u\|_2^2\lam^2} du\\
& =  C \Big(\frac{n}{t_0^2}\Big) \int_{\|u\|_2 \le r_0}  e^{-(\frac{n}{2 \lam t_0^2} + c'n) \|u\|_2^2\lam^2} du= O(\lam^{-2}\delta^{2}),
\end{align*}
where we used change of variable $v=\sqrt{\lam n } u$ and that $\int_v e^{-c\|v\|_2^2} dv =O(1)$ in the last estimate.

For $J_2$ we trivially have
\begin{align*}J_2 &=  C \Big(\frac{n}{t_0^2}\Big) \int_{r_0 \le \|u\|_2}  \prod_{j=1}^n |\bE e^{i \xi_j \langle (a_j,b_j),u \rang}|  e^{-\frac{n \|u\|_2^2}{2 t_0^2}} du  \le  C \delta^2 n \int_{r_0 \le \|u\|_2} e^{-\delta^2 n \|u\|_2^2/2}du \\
& \le C \delta^2 \int_{r_0 \sqrt{n} \le \|v\|_2} e^{-\delta^2 \|v\|_2^2} dv = O(\delta^2)
\end{align*}
provided that $\delta \gg n^{-1/2}$.

\section{Proof of Lemma \ref{lemma:good_events_are_well_separated}}  \label{appendix:separation}
\begin{proof}[Proof of part (1)]
	Fix $\al \in [N]$ with $\theta_\al\not\in \badarcs$ and assume that the event $\cA_\al$ holds. Recall that $F_\al$ given by \eqref{eq:def_of_linear_approximation} is the affine transformation which locally approximates the polynomial $\frac{1}{\sqrt{n}}f$ inside $C_\al$. It is evident from its definition up to the translation term $(X(\ang_\al),Y(\ang_\al))$, it is an orthogonal transformation (in particular it maps disks to disks; in fact this it not a coincidence, and can be seen as a consequence of the Cauchy--Riemann equation for $f$ in the polar coordinates.)
	
	On the event $\cG$, we have the bound
	\begin{align}
	\label{eq:comparison_between_F_al_and_F_al_1_second_derivative}
	\left|F_\al(z) - F_{\al+1}(z)\right| \leq \left|F_\al(z) - \frac1{\sqrt{n}}f(z)\right| &+ \left|F_{\al+1}(z) - \frac1{\sqrt{n}}f(z)\right| \nonumber \\ &\ll \frac{n^2 \log^2n}{N^2} \ll \frac{\log^{2K_0+2}n}{n^2}
	\end{align}
	valid for all $z\in C_{\al+1}$. On the event $\cA_\al$, assume that $\left\{\cA_\al, \ |\dang_\al - \ang_{\al}| \leq \frac{\pi}{2N} - \frac{\pi}{2N\log^{K_0}n} \right\}$. We have
	\begin{equation*}
	\text{dist}\left((1+\rho_\al) e^{i\ang_\al}, C_{\al+1}\right) \gg \frac{1}{N\log^{K_0}n}
	\end{equation*}
	which, together with the fact that $F_\al(\rho_\al,\dang_\al) = 0$ and the bound \eqref{eqn:F:nice}, implies that for all $z\in C_{\al+1}$ we have
	\begin{align*}
	|F_{\al+1}(z)| &\ge |F_\al(z)| - \left|F_\al(z) - F_{\al+1}(z)\right|\gg \frac{n \log^{-6K_0}n}{N\log^{K_0}n} - \frac{\log^{2K_0+2}n}{n^2} \gg \frac{n}{N\log^{7K_0}n}
	\end{align*}
	where in the second inequality we used \eqref{eq:comparison_between_F_al_and_F_al_1_second_derivative} and \eqref{eqn:F:nice}. In particular, we see that $F_{\al+1}$ does not vanish in $C_{\al+1}$ which implies that $\cA_{\al+1}$ does not hold.
\end{proof}
\begin{proof}[Proof of part (2)]
	We argue similarly as in the proof of the first part, only that we do not need to impose the extra separation within $C_\al$. Fix distinct $\al, \al' \not\in \badarcs$ such that $D:=|\ang_\al - \ang_{\al^\prime}| \in (\pi/N, n^{-1}\log^{-4K_0}n]$. Assuming $\cA_\al\cap \cG$ holds we have
	\begin{equation*}
	|F_\al(z) - F_{\al^\prime}(z)| \leq \left|\frac1{\sqrt{n}}f(z) - F_{\al}(z)\right| + \left| \frac1{\sqrt{n}}f(z) - F_{\al^\prime}(z)\right| \ll n^{2}(\log^2n) D^2.
	\end{equation*} 
	for all $z\in C_{\al^\prime}$. Furthermore, since $\frac{1}{\sqrt{n}}|f^\prime(e^{i\ang_\al})|\geq n\log^{-2K_0}n$ (by the Cauchy--Riemann equations in polar form -- see \Cref{rmk:CR}) we know that for all such $z$,
	\[
	|F_\al(z)| \gg |z-(1+\rho_\al)e^{i\ang_\al}||\frac1{\sqrt{n}}f^\prime(e^{i\ang_\al})| \gg D n\log^{-2K_0}n
	\]
	which implies that
	\begin{equation}
	\label{eq:lower_bound_on_F_al_prime}
	|F_{\al^\prime}(z)| \gg D n\log^{-2K_0}n - D^2n^{2}\log^2n \gg D n\log^{-2K_0}n
	\end{equation}
	since $D\le n^{-1}\log^{-4K_0}n$. 
	Since \eqref{eq:lower_bound_on_F_al_prime} holds for all $z\in C_{\al^\prime}$ we see that $F$ cannot vanish there, and hence $\cA_{\al^\prime}$ does not hold.
\end{proof}

 	\bigskip
\end{document}